\documentclass[review]{elsarticle}
\makeatletter
\def\ps@pprintTitle{%
	\let\@oddhead\@empty
	\let\@evenhead\@empty
	\def\@oddfoot{\centerline{\thepage}}%
	\let\@evenfoot\@oddfoot}
\makeatother
\usepackage{lineno,hyperref}
\modulolinenumbers[5]

\usepackage{amsmath,amssymb,epsfig,amsthm,bm}%bbm,
\usepackage{a4wide,xcolor,eucal,enumerate,mathrsfs,graphics,caption,subfig,booktabs,verbatim}
\usepackage{dsfont,natbib}

\biboptions{authoryear}
\journal{}

\theoremstyle{definition}
\newtheorem{de}{Definition}[section]

\theoremstyle{plain}
\newtheorem{theo}[de]{Theorem}
\newtheorem{lemma}[de]{Lemma}

\theoremstyle{remark}

\newcommand{\R}{\mathbb{R}}
\newcommand{\Z}{\mathbb{Z}}
\newcommand{\p}{\mathbb{P}}

\newcommand{\F}{\mathcal{F}}

\def\argmin{\mathop{\text{argmin}}}

\begin{document}
	
	\begin{frontmatter}
		
		\title{Smooth estimation of a monotone hazard and a monotone density under random censoring}
		%\tnotetext[mytitlenote]{Fully documented templates are available in the elsarticle package on \href{http://www.ctan.org/tex-archive/macros/latex/contrib/elsarticle}{CTAN}.}
		
		%% Group authors per affiliation:
		%\author{Hendrik P.~Lopuha\"a and Eni Musta\fnref{myfootnote}}
		%\address{DIAM, Faculty EEMCS, Delft University of Technology, Mekelweg 4, 2628CD, Delft, The Netherlands}
		%\fntext[myfootnote]{Since 1880.}
		
		%% or include affiliations in footnotes:
		\author[mymainaddress]{Hendrik P.~Lopuha\"a}
		
		\author[mymainaddress]{Eni Musta\corref{mycorrespondingauthor}}
		\cortext[mycorrespondingauthor]{Corresponding author}
		\ead{e.musta@tudelft.nl}
		
		\address[mymainaddress]{DIAM, Faculty EEMCS, Delft University of Technology, Mekelweg 4, 2628 CD Delft, The Netherlands}

\begin{abstract}
We consider kernel smoothed Grenander-type estimators for a monotone hazard rate
and a monotone density in the presence of randomly right censored data.
We show that they converge at rate $n^{2/5}$ and that the limit distribution at a fixed point is Gaussian with explicitly given mean and variance.
It is well-known that standard kernel smoothing leads to inconsistency problems at the boundary points.
It turns out that, also by using a boundary correction, we can only establish uniform consistency on intervals that stay away from the end point of the support (though we can go arbitrarily close to the right boundary).
\end{abstract}

\end{frontmatter}
\section{Introduction}
Nonparametric estimation under shape constraints is currently a very active research area in statistics.
A frequently encountered problem in this field  is the estimation of the hazard rate, which, in survival analysis, is defined as the probability that an individual will experience an event within a small time interval given that the subject has survived until the beginning  of this interval. In this context, monotonicity constraints arise naturally reflecting the property of aging or becoming more reliable as the survival time increases.
Beside this, it might be also of interest to characterize the distribution of the event times in terms of the density assuming that it is monotone.

Popular estimators, such as the nonparametric maximum likelihood  estimator (NPMLE) or
Grenander type estimators, are typically piecewise constant and converge at rate~$n^{1/3}$.
However, at the price of additional smoothness assumptions on the hazard or density function, the cube-root-$n$ rate of convergence can be improved. 
Smooth estimation has received considerable attention in the literature, because it is needed to prove that a bootstrap method works
(see for instance, \citet{kos2008}; \citet{SBW}).
Moreover, it provides a straightforward estimate of the derivative of the function of interest, which is of help when constructing confidence intervals (see for instance,  \citet{nane15}).
{\citet{HM88} is one of the earliest papers which combines  smoothness and isotonization for nonparametric regression procedures.} Various approaches can be used to obtain smooth shape constrained estimators.
It essentially depends on the methods of both isotonization and smoothing and on the order of operations (see for instance, \citet{mammen1991} or \citet{GJW10}). 
Chapter~8 in \citet{GJ14} gives an overview of such methods.

In this paper, we focus on kernel smoothed Grenander-type estimator (SG)
of the hazard function and the probability density in the presence of randomly right censored data.
\citet{huang-wellner1995} consider the random censorship model and
Grenander estimators of a monotone hazard and density,
obtained by taking slopes of the greatest convex minorant (lowest concave majorant) of the Nelson-Aalen
or Kaplan-Meier estimator.
Consistency and asymptotic distribution are established, together with the asymptotic equivalence with the maximum likelihood estimator. 
The same model and the $L_p$-error of this type of estimators was investigated in~\citet{durot2007}, { while pointwise confidence intervals for a monotone hazard rate via inversion of the likelihood ratio statistic were proposed in \citet{MB08}.}
However, these papers do not take in consideration smoothing options. On the other hand, the kernel smoothed Grenander-type estimator of a monotone hazard in the context of the Cox model, which is a generalization of the right censoring that takes into account covariates, was introduced  in \citet{Nane}, but without further development of its asymptotic distribution. However, on the basis of Theorem 3.1  in \citet{GJ13}, where no censoring takes place, our main result (Theorem~\ref{theo:distr}) was conjectured by \citet{Nane}. Afterwards, Theorem 11.8 in \citet{GJ14} states the limit distribution of the smoothed maximum likelihood estimator (SMLE) of a monotone hazard function using a more delicate argument.
Hence, it seems quite natural to address the problem of the smoothed Grenander-type estimator.
The present paper, aims at giving a rather short and direct proof of its limit distribution, relying on the method developed in \citet{GJ13} together with a Kiefer-Wolfowitz type of result derived in~\citet{DL14}.
Both Theorem~\ref{theo:distr} and Theorem~\ref{theo:distr-dens}, highlight the fact that also after applying smoothing techniques, the NPMLE and the Grenander estimator remain asymptotically equivalent.

Furthermore, we study inconsistency problems at the boundaries of the support.
In order to prevent those, different approaches have been tried, including penalization (see for instance, \citet{GJ13}) and boundary corrections
(see for instance, \citet{Albers}).
However, no method performs strictly better than the others.
We choose to use boundary kernels, but we discover that still the inconsistency at the right boundary can not be avoided.
The main reason for this is that a bound on the distance between the cumulative hazard (cumulative distribution) function and the Nelson-Aalen
(Kaplan-Meier) estimator is only available on intervals strictly smaller than the end point of the support.

The paper is organized as follows.
In Section~\ref{sec:model} we briefly introduce the Grenander estimator in the random censorship model and recall some results needed in the sequel.
The smoothed estimator of a monotone hazard function is described in Section~\ref{sec:asymptotics} and it is shown to be asymptotically normally distributed.
Moreover, a smooth estimator based on boundary kernels is studied and uniform consistency is derived.
Using the same approach, in Section~\ref{sec:dens} we deal with the problem of estimating a smooth monotone density function.
Section~\ref{sec:conf-int} is devoted to numerical results on pointwise confidence intervals.
Finally, we end with a short discussion on how these results relate to a more general picture.

\section{The random censorship model}
\label{sec:model}

Suppose we have an i.i.d.~sample $X_1,\dots,X_n$ with distribution function $F$ and density~$f$, representing the survival times.
Let $C_1,\dots,C_n$ be the i.i.d.~censoring variables with a distribution function $G$ and density~$g$.
Under the random censorship model, we assume that the survival time $X$ and the censoring time $C$ are independent and the observed data consists of
i.i.d.~pairs
of random variables $(T_1,\Delta_1),\dots,(T_n,\Delta_n)$, where $T$ denotes the follow-up time~$T=\min(X,C)$
and $\Delta=\mathds{1}_{\{X\leq C\}}$ is the censoring indicator.

Let $H$ and $H^{uc}$ denote the distribution function of the follow-up time and the sub-distribution function of the uncensored observations, respectively, i.e.,
$H^{uc}(x)=\p(T\leq x,\Delta=1)$.
Note that $H^{uc}(x)$ and $H(x)$ are differentiable with derivatives
\[
h^{uc}(x)=f(x)\left(1-G(x)\right)
\]
and
\[
h(x)=f(x)(1-G(x))+g(x)(1-F(x))
\]
respectively.
We also  assume that $\tau_H=\tau_G<\tau_F\leq\infty$, where $\tau_F,\,\tau_G$ and $\tau_H$ are the end points of the support of $F,\,G$ and $H$.

The hazard rate $\lambda$ is characterized by the following relation
\[
\lambda(t)=\frac{f(t)}{1-F(t)}
\]
and we refer to the quantity
\[
\Lambda(t)=\int_0^t\lambda(u)\,\mathrm{d}u,
\]
as the cumulative hazard function.
First, we aim at estimating $\lambda$, subject to the constraint that it is increasing (the case of a decreasing hazard is analogous), on the basis of $n$ observations $(T_1,\Delta_1),\dots,(T_n,\Delta_n).$
The Grenander-type estimator $\tilde{\lambda}_n$ of $\lambda$ is defined as the left-hand slope of the greatest convex minorant $\tilde{\Lambda}_n$ of the Nelson-Aalen estimator $\Lambda_n$
of the cumulative hazard function $\Lambda$, where
\[
\Lambda_n(t)=
\sum_{i=1}^n
\frac{\mathds{1}_{\{T_i\leq t\}}\Delta_i}{\sum_{j=1}^n \mathds{1}_{\{T_j\geq T_i\}}}.
\]
Figure~\ref{fig:NA} shows the Nelson-Aalen estimator and its greatest convex minorant for a sample of $n=500$ from
a Weibull distribution with shape parameter $3$ and scale parameter $1$ for the event times and the uniform distribution on~$(0,1.3)$ for the censoring times.
We consider only the data up to the last observed time before the $90\%$ quantile of $H$.
The resulting Grenander-type estimator can be seen in Figure~\ref{fig:gren}.
\begin{figure}[t]
\includegraphics[width=.7\textwidth]{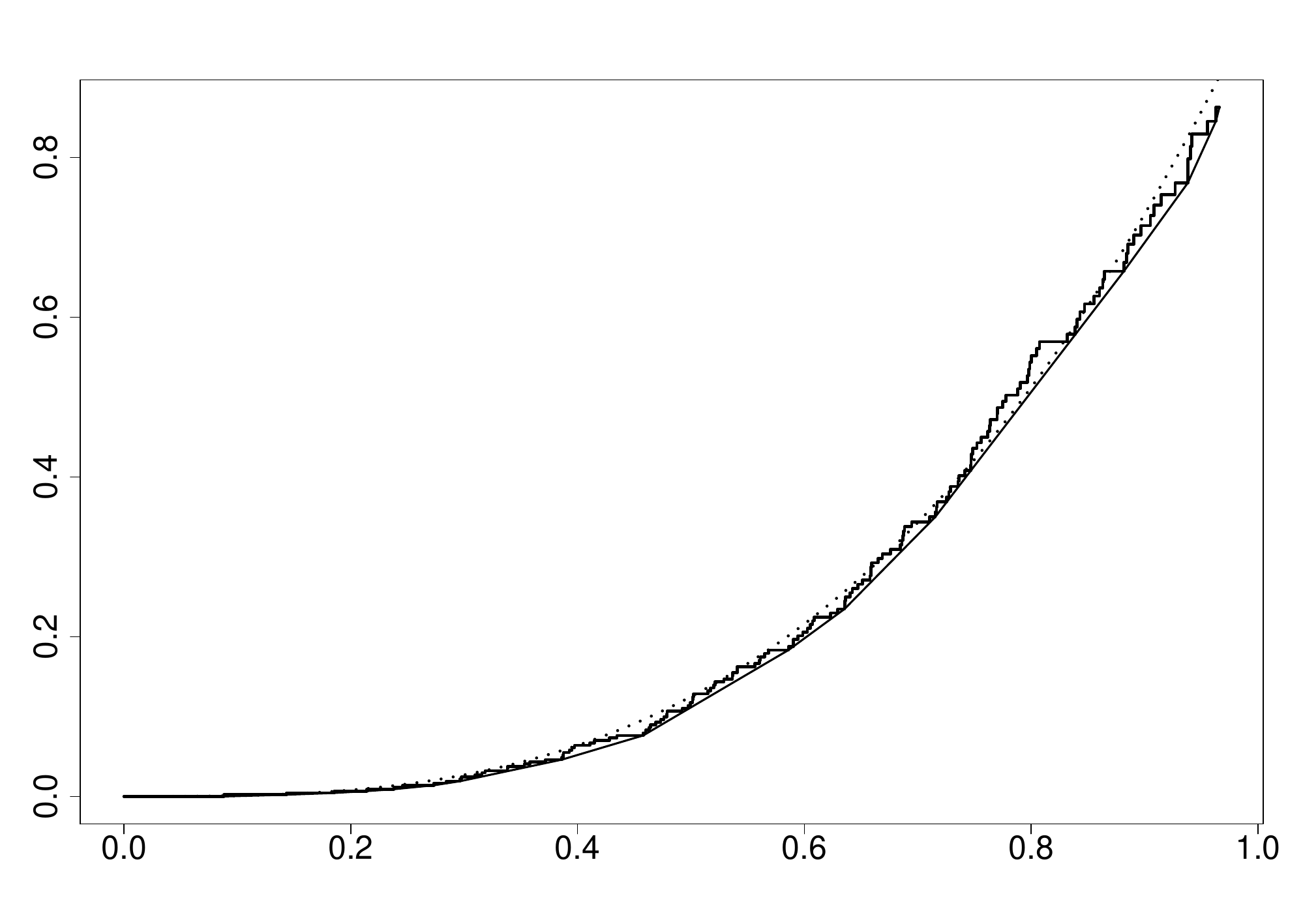}
\caption{The Nelson-Aalen estimator (piecewise constant solid line) of the cumulative hazard (dotted line) and its greatest convex minorant (solid line).}
\label{fig:NA}
\end{figure}%

In \citet{huang-wellner1995} it is shown that the Grenander estimator of a nondecreasing hazard rate satisfies the following pointwise consistency result
\begin{equation}
\label{eqn:consG}
\lambda(t-)\leq \liminf_{n\to \infty} \tilde{\lambda}_n(t)\leq \limsup_{n\to \infty} \tilde{\lambda}_n(t)\leq  \lambda(t+),
\end{equation}
with probability one and for all $0<t<\tau_H$, where $\lambda(t-)$ and $\lambda(t+)$ denote the left and right limit at $t$.
Moreover, we will also make use of the fact that for any $0<M<\tau_H$,
\begin{equation}
\label{eqn:chbound}
\sqrt{n}
\sup_{u\in[0,M]}
\left|\tilde{\Lambda}_n(u)-\Lambda(u)\right|
=
O_P(1),
\end{equation}
(see for instance, \citet{LopuhaaNane2013}, Theorem 5, in the case $\beta=0$, or \citet{VW96}, Example~3.9.19).

It becomes useful to introduce
\begin{equation}
\label{def:Phi}
\Phi(x)=
\int \mathds{1}_{[x,\infty)}(y)\,\mathrm{d}\p(y,\delta)=1-H(x)
\end{equation}
and
\begin{equation}
\label{def:Phin}
\Phi_n(x)=\int \mathds{1}_{[x,\infty)}(y)\,\mathrm{d}\p_n(y,\delta),
\end{equation}
where $\p$ is the probability distribution of $(T,\Delta)$ and $\p_n$ is the empirical measure of the pairs
$(T_i,\Delta_i)$, $i=1,\dots,n$. From Lemma~4 in~\citet{LopuhaaNane2013} we have,
\begin{equation}
\label{eqn:phi}
\sup_{x\in [0,\tau_H]}
|\Phi_n(x)-\Phi(x)|
\to0, \text{ a.s., and }
\sqrt{n}\sup_{x\in[0,\tau_H]}|\Phi_n(x)-\Phi(x)|=O_P(1).
\end{equation}
Let us notice that, with these notations, we can also write
\begin{equation}
\label{eqn:cum-haz}
\Lambda_n(t)
=
\int\frac{\delta\mathds{1}_{\{u\leq t\}}}{\Phi_n(u)}\,\mathrm{d}\p_n(u,\delta),
\qquad
\Lambda(t)
=
\int\frac{\delta\mathds{1}_{\{u\leq t\}}}{\Phi(u)}\,\mathrm{d}\p(u,\delta).
\end{equation}
Our second objective is to estimate a monotone (e.g., increasing) density function $f$ .
In this case the Grenander-type estimator $\tilde{f}_n$ of $f$ is defined as the left-hand slope of the greatest convex minorant $\tilde{F}_n$ of the Kaplan-Meier estimator $F_n$ of the cumulative distribution function~$F$.
Pointwise consistency of  the Grenander estimator of a nondecreasing density:
\begin{equation}
\label{eqn:consGdens}
f(t-)\leq \liminf_{n\to \infty} \tilde{f}_n(t)\leq \limsup_{n\to \infty} \tilde{f}_n(t)\leq  f(t+),
\end{equation}
with probability one, for all $0<t<\tau_H$, where $f(t-)$ and $f(t+)$ denote the left and right limit at $t$, is proved in \citet{huang-wellner1995}.
Moreover, for any $0<M<\tau_H$, it holds
\begin{equation}
\label{eqn:cdbound}
\sqrt{n}
\sup_{u\in[0,M]}
\left|\tilde{F}_n(u)-F(u)\right|
=
O_P(1),
\end{equation}
(see for instance, \citet{BC74}, Theorem 5). By Theorem 2 in~\citet{MR88}, for each $0<M<\tau_H$ and $x\geq 0$, we have the following strong approximation
\begin{equation}
\label{eqn:approx}
\p\left\{\sup_{t\in[0,M]}n\left|F_n(t)-F(t)-n^{-1/2}(1-F(t))W\circ L(t)\right|>x+K_1\log n\right\}\leq K_2\mathrm{e}^{-K_3x},
\end{equation}
where $K_1,\,K_2,\,K_3$ are positive constants, $W$ is a Brownian motion and
\begin{equation}
\label{eqn:L}
L(t)=\int_0^t \frac{\lambda(u)}{1-H(u)}\,\mathrm{d}u.
\end{equation}
\section{Smoothed Grenander-type estimator of a monotone hazard}
\label{sec:asymptotics}
Next, we introduce the smoothed Grenander-type estimator $\tilde{\lambda}_n^{SG}$ of an increasing hazard.
Kernel smoothing is a rather simple and broadly used method.
Let $k$ be a standard kernel, i.e.,
\begin{equation}
\label{def:kernel}
\text{$k$ is a symmetric probability density with support $[-1,1]$.}
\end{equation}
We will consider the scaled version
\[
k_b(u)=\frac{1}{b}k\left(\frac{u}{b}\right)
\]
of the kernel function $k$, where $b=b_n$ is a bandwidth that depends on the sample size, such that
\begin{equation}
\label{eqn:band}
0<b_n\to 0\quad\text{and}\quad nb_n\to\infty, \quad\text{as }n\to\infty.
\end{equation}
From now on we will use the notation $b$ instead of $b_n$.

For a fixed $x\in[0,\tau_H]$, the smoothed Grenander-type estimator $\tilde{\lambda}_n^{SG}$ is defined by
\begin{equation}
\label{eqn:SG}
\tilde{\lambda}_n^{SG}(x)=\int_{(x-b)\vee 0}^{(x+b)\wedge\tau_H} k_b(x-u)\,\tilde{\lambda}_n(u)\,\mathrm{d}u=\int k_b(x-u)\,\mathrm{d}\tilde{\Lambda}_n(u).
\end{equation}
Figure~\ref{fig:gren} shows the Grenander-type estimator together with the kernel smoothed version for the same sample as in Figure~\ref{fig:NA}.
We used the triweight kernel function
\[
k(u)=\frac{35}{32}(1-u^2)^3\mathds{1}_{\{|u|\leq 1\}}
\]
and the bandwidth $b=c_{opt}\,n^{-1/5}$, where $c_{opt}$ is the asymptotically MSE-optimal constant (see \eqref{eqn:c_opt}) calculated in the point $x_0=0.5$. {Actually, the choice of the kernel function does not seem to effect the results.}
\begin{figure}[t]
\includegraphics[width=.7\textwidth]{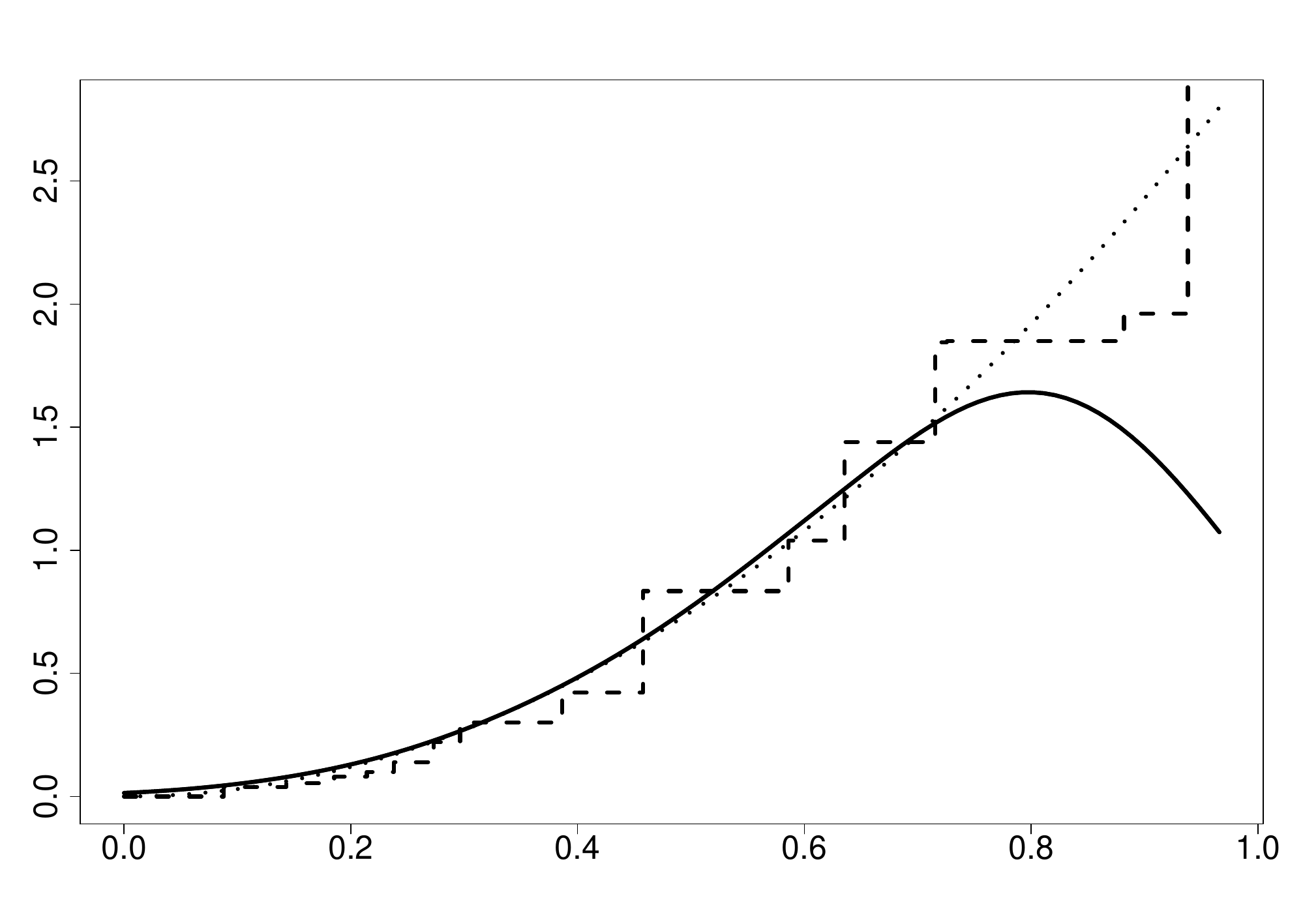}
\caption{The Grenander-type estimator (dashed line) of the hazard rate (dotted line) and the kernel smoothed version (solid line).}
\label{fig:gren}
\end{figure}%

The following result is rather standard when dealing with kernel smoothed isotonic estimators
(see for instance, \citet{Nane},Chapter 5).
For completeness, we provide a rigorous proof.
\begin{theo}
\label{theo:cons}
Let $k$ be a kernel function satisfying~\eqref{def:kernel} and let
$\tilde{\lambda}^{SG}$ be the smoothed Grenander-type estimator defined in~\eqref{eqn:SG}. Suppose that the hazard function $\lambda$ is nondecreasing and continuous. Then for each $0<\epsilon<\tau_H$, it holds
\[
\sup_{x\in[\epsilon,\tau_H-\epsilon]}|\tilde{\lambda}_n^{SG}(x)-\lambda(x)|\to 0
\]
with probability one.
\end{theo}
\begin{proof}
First, note that for a fixed $x\in(0,\tau_H)$ and sufficiently large  $n$, we have $0<x-b< x+b<\tau_H.$
We start by writing
\[
\tilde{\lambda}_n^{SG}(x)-\lambda(x)
=
\tilde{\lambda}_n^{SG}(x)-\tilde{\lambda}_n^s(x)
+
\tilde{\lambda}_n^{s}(x)-\lambda(x),
\]
where $\tilde{\lambda}^s_n(x)=\int k_b(x-u)\,\lambda(u)\,\mathrm{d}u$.
Then, a change of variable yields
\[
\tilde{\lambda}^{s}_n(x)-\lambda(x)
=
\int_{-1}^1 k(y)
\left\{\lambda(x-by)-\lambda(x)\right\}\,\mathrm{d}y.
\]
Using the continuity of $\lambda$ and applying the dominated convergence theorem, we obtain that,
for each $x\in (0,\tau_H)$,
\begin{equation}
\label{eqn:partcons}
\tilde{\lambda}^s_n(x)\to\lambda(x),
\quad
\text{as }n\to\infty.
\end{equation}
On the other hand,
\[
\tilde{\lambda}^{SG}_n(x)-\tilde{\lambda}^s_n(x)
=
\int_{-1}^1 k(y)
\left\{
\tilde{\lambda}_n(x-by)-\lambda(x-by)
\right\}\,\mathrm{d}y.
\]
Choose $\epsilon>0$.
Then by continuity of $\lambda$, we can find $\delta>0$,
such that $0<x-\delta<x+\delta<\tau_H$ and $|\lambda(x+\delta)-\lambda(x-\delta)|<\epsilon$.
Then, there exists $N$ such that, for all $n\geq N$ and for all $y\in[-1,1]$, it holds $|by|<\delta$.
Hence, by the monotonicity of the hazard rate, we get
\[
\tilde{\lambda}_n(x-\delta)-\lambda(x+\delta)\leq \tilde{\lambda}_n(x-by)-\lambda(x-by)\leq \tilde{\lambda}_n(x+\delta)-\lambda(x-\delta).
\]
It follows from~\eqref{eqn:consG} and~\eqref{def:kernel} that
\[
-\epsilon
\leq
\liminf_{n\to\infty}
\tilde{\lambda}^{SG}_n(x)-\tilde{\lambda}^s_n(x)
\leq
\limsup_{n\to\infty}
\tilde{\lambda}^{SG}_n(x)-\tilde{\lambda}^s_n(x)
\leq
\epsilon,
\]
with probability one.
Since $\epsilon>0$ is arbitrary, together with~\eqref{eqn:partcons}, this proves the strong pointwise consistency
at each fixed $x\in(0,\tau_H)$.
Finally, uniform consistency in $[\epsilon,\tau_H-\epsilon]$ follows from the fact that we have a sequence of monotone functions converging pointwise to a continuous, monotone function on a compact interval.
\end{proof}
It is worth noticing that, if one is willing to assume that $\lambda$ is twice differentiable with uniformly bounded first and second derivatives,
and that $k$ is differentiable with a bounded derivative, then we get a more precise result on the order of convergence
\[
\sup_{x\in[\epsilon,\tau_H-\epsilon]}|\tilde{\lambda}_n^{SG}(x)-\lambda(x)|=O_P(b^{-1}n^{-1/2}+b^2).
\]
Such extra assumptions are considered in Theorem 5.2 in \citet{Nane} for the Cox model and the right censoring model is just a particular case with regression parameters $\beta=0$. Furthermore, in a similar way, it can be proved that also the estimator for the derivative of the hazard is uniformly consistent in $[\epsilon,\tau_H-\epsilon]$, provided that $\lambda$ is continuously differentiable and the kernel is differentiable with bounded derivative.

The pointwise asymptotic normality of the smoothed Grenander estimator is stated in the next theorem.
Its proof is inspired by the approach used in \citet{GJ13}.
The key step consists in using a Kiefer-Wolfowitz type of result for the Nelson-Aalen estimator,
which has recently been obtained by~\citet{DL14}.
\begin{theo}
\label{theo:distr}
Let $\lambda$ be a nondecreasing and twice continuously differentiable hazard such that $\lambda$ and~$\lambda'$ are strictly positive.
Let $k$ satisfy~\eqref{def:kernel} and suppose that it is
differentiable with a uniformly bounded derivative.
If $bn^{1/5}\to c\in(0,\infty)$, then for each $x\in(0,\tau_h)$,
\[
n^{2/5}
\big(
\tilde{\lambda}_n^{SG}(x)-\lambda(x)
\big)
\xrightarrow{d}
N(\mu,\sigma^2),
\]
where
\begin{equation}
\label{eqn:mu-sigma}
\mu=\frac{1}{2}c^2\lambda''(x)\int u^2k(u)\,\mathrm{d}u
\quad\text{ and }\quad
\sigma^2=
\frac{\lambda(x)}{c\,(1-H(x))}\int k^2(u)\,\mathrm{d}u.
\end{equation}
For a fixed $x\in(0,\tau_h)$, the asymptotically MSE-optimal bandwidth $b$ for $\tilde{\lambda}^{SG}$ is given by
$c_{opt}(x)n^{-1/5}$, where
\begin{equation}
\label{eqn:c_opt}
c_{opt}(x)=\left\{\lambda(x)\int k^2(u)\,\mathrm{d}u \right\}^{1/5}\left\{(1-H(x))\lambda''(x)^2 \left(\int u^2\,k(u)\,\mathrm{d}u\right)^2\right\}^{-1/5}.
\end{equation}
\end{theo}
\begin{proof}
Once again we fix $x\in(0,\tau_H)$. Then, for sufficiently large  $n$, we have $0<x-b< x+b\leq M<\tau_H$.
We write
\begin{equation}
\label{eqn:distr1}
\begin{split}
\tilde{\lambda}^{SG}_n(x)
&
=\int k_b(x-u)\,\mathrm{d}\Lambda(u)+\int k_b(x-u)\,\mathrm{d}(\Lambda_n-\Lambda)(u)\\
&\quad+\int k_b(x-u)\,\mathrm{d}(\tilde{\Lambda}_n-\Lambda_n)(u).
\end{split}
\end{equation}
The first (deterministic) term on the right hand side of~\eqref{eqn:distr1} gives us
the asymptotic bias by using a change of variables, a Taylor expansion, and the properties of the kernel:
\[
\begin{split}
n^{2/5}
\left\{
\int_{x-b}^{x+b} k_b(x-u)\,\lambda(u)\,\mathrm{d}u-\lambda(x)
\right\}
&=
n^{2/5}
\int_{-1}^1 k(y)\,
\left\{
\lambda(x-by)-\lambda(x)
\right\}\,\mathrm{d}y\\
&=
n^{2/5}
\int_{-1}^1 k(y)
\left\{
-\lambda'(x)by+\frac12\lambda''(\xi_n)b^2y^2
\right\}\,\mathrm{d}y,\\
&\to
\frac{1}{2}\,c^2\,\lambda''(x)
\int_{-1}^1 y^2\,k(y)\,\mathrm{d}y.
\end{split}
\]
where $|\xi_n-x|<b|y|\leq b\to0$.
On the other hand, the last term on the right hand side of~\eqref{eqn:distr1} converges to $0$ in probability.
Indeed, integration by parts formula enables us to write
\[
\begin{split}
n^{2/5}\int_{x-b}^{x+b}k_b(x-u)\,\mathrm{d}(\tilde{\Lambda}_n-\Lambda_n)(u)
&=
n^{2/5}
\int_{x-b}^{x+b}
\left\{
\tilde{\Lambda}_n(u)-\Lambda_n(u)
\right\}
\frac{1}{b^2}\,k'\left(\frac{x-u}{b}\right)\,\mathrm{d}u\\
&=
\frac{n^{2/5}}{b}\int_{-1}^1
\left\{
\tilde{\Lambda}_n(x-by)-\Lambda_n(x-by)
\right\}
k'(y)\,\mathrm{d}y,
\end{split}
\]
and then we use $\sup_{t\in[0,\tau_H]}|\tilde{\Lambda}_n(t)-\Lambda_n(t)|=O_p(n^{-2/3}(\log n)^{2/3})$ (see \citet{DL14}, Corollary 3.4) together with the boundedness of $k'$.

What remains is to prove that
\[
n^{2/5}\,\int k_b(x-u)\,\mathrm{d}(\Lambda_n-\Lambda)(u)\xrightarrow{d} N(0,\sigma^2),
\]
where $\sigma^2$ is defined in~\eqref{eqn:mu-sigma}. Let us start by writing
\[
n^{2/5}\int_{x-b}^{x+b} k_b(x-u)\,\mathrm{d}(\Lambda_n-\Lambda)(u)=\frac{1}{\sqrt{bn^{1/5}}}\int_{-1}^1 k(y)\,\mathrm{d}\hat{W}_n(y),
\]
where, for each $y\in[-1,1]$, we define
\begin{equation}
\label{eqn:distr2}
\begin{split}
\hat{W}_n(y)
&=
\sqrt{\frac{n}{b}}
\left\{
\Lambda_n(x-by)-\Lambda_n(x)-\Lambda(x-by)+\Lambda(x)
\right\}\\
&=
\sqrt{\frac{n}{b}}
\int
\frac{\delta}{\Phi_n(u)}\left\{\mathds{1}_{[0,x-by]}(u)-\mathds{1}_{[0,x]}(u)\right\}\,\mathrm{d}\p_n(u,\delta)\\
&\qquad-
\sqrt{\frac{n}{b}}
\int\frac{\delta}{\Phi(u)}\left\{\mathds{1}_{[0,x-by]}(u)-\mathds{1}_{[0,x]}(u)\right\}\,\mathrm{d}\p(u,\delta)\\
&=
b^{-1/2}
\int
\frac{\delta}{\Phi(u)}\left\{\mathds{1}_{[0,x-by]}(u)-\mathds{1}_{[0,x]}(u)\right\}\,\mathrm{d}\sqrt{n}(\p_n-\p)(u,\delta)\\
&\qquad+
\sqrt{\frac{n}{b}}
\int \delta\left\{\mathds{1}_{[0,x-by]}(u)-\mathds{1}_{[0,x]}(u)\right\}
\left\{
\frac{1}{\Phi_n(u)}-\frac{1}{\Phi(u)}
\right\}\,\mathrm{d}\p_n(u,\delta).
\end{split}
\end{equation}
Here we took advantage of the representations in~\eqref{eqn:cum-haz}.
The  last term in the right hand side of~\eqref{eqn:distr2} is bounded in absolute value by
\[
\sqrt{\frac{n}{b}}
\frac{1}{\Phi_n(M)\,\Phi(M)}
\int \delta
\left|\mathds{1}_{[0,x-by]}(u)-\mathds{1}_{[0,x]}(u)\right|
\left|\Phi(u)-\Phi_n(u)\right|\,
\mathrm{d}\p_n(u,\delta)=o_P(1).
\]
Indeed, by using~\eqref{eqn:phi}, we obtain that $1/\Phi_n(M)=O_P(1)$ and then it suffices to prove that
\[
b^{-1/2}
\int
\delta\left|\mathds{1}_{[0,x-by]}(u)-\mathds{1}_{[0,x]}(u)\right|
\,\mathrm{d}\p_n(u,\delta)=o_P(1).
\]
To do so, we write the left hand side as
\begin{equation}
\label{eqn:decompose}
\begin{split}
&
b^{-1/2}\,\int \delta\left|\mathds{1}_{[0,x-by]}(u)-\mathds{1}_{[0,x]}(u)\right|\,\mathrm{d}\p(u,\delta)\\
&\quad+
b^{-1/2}\,\int \delta\left|\mathds{1}_{[0,x-by]}(u)-\mathds{1}_{[0,x]}(u)\right|\,\mathrm{d}(\p_n-\p)(u,\delta)\\
&=
b^{-1/2}\,\big|H^{uc}(x-by)-H^{uc}(x)\big|+O_p(b^{-1/2}n^{-1/2})
=
o_P(1).
\end{split}
\end{equation}
Here we have used that $H^{uc}$ is continuously differentiable
and that the class of indicators of intervals forms a VC-class, and is therefore Donsker
(see \citet{VW96}, Theorem 2.6.7 and Theorem 2.5.2).

The last step consists in showing that
\[
b^{-1/2}
\int
\frac{\delta}{\Phi(u)}\left\{\mathds{1}_{[0,x-by]}(u)-\mathds{1}_{[0,x]}(u)\right\}\,\mathrm{d}\sqrt{n}(\p_n-\p)(u,\delta)
\xrightarrow{d}
\sqrt{\frac{\lambda(x)}{1-H(x)}}\,W(y),
\]
where $W$ is a two sided Brownian motion. This follows from Theorem~2.11.23 in \citet{VW96}.
Indeed, we can consider the functions
\[
f_{n,y}(u,\delta)
=
b^{-1/2}\frac{\delta}{\Phi(u)}
\left\{
\mathds{1}_{[0,x-by]}(u)-\mathds{1}_{[0,x]}(u)
\right\},
\qquad y\in[-1,1].
\]
with envelopes $F_n(u,\delta)=b^{-1/2}\Phi(M)^{-1}\delta\mathds{1}_{[x-b,x+b]}(u)$.
It can be easily checked that
\[
\Vert F_n\Vert_{L_2(\p)}^2=\frac{1}{b\,\Phi^2(M)}\int_{x-b}^{x+b} f(u)(1-G(u))\,\mathrm{d}u=O(1),
\]
and that for all $\eta>0$,
\[
\frac{1}{b\Phi^2(M)}
\int_{\left\{u:b^{-1/2}\Phi(M)^{-1}\mathds{1}_{[x-b,x+b]}(u)>\eta\sqrt{n}\right\}}
f(u)(1-G(u))\,\mathrm{d}u
\to0.
\]
Moreover, for every sequence $\delta_n\downarrow0$, we have
\[
\frac{1}{b\Phi^2(M)}
\sup_{|s-t|<\delta_n}
\int_{x-b(s\vee t)}^{x-b(s\wedge t)}
f(u)(1-G(u))\,\mathrm{d}u
\to0.
\]
Since $f_{n,y}$ are sums and products of bounded monotone functions, the bracketing number is bounded
(see \citet{VW96}, Theorem 2.7.5)
\[
\log N_{[\,]}\left(\epsilon\Vert F_n\Vert_{L_2(\p)},\F_n, \Vert \cdot\Vert_{L_2(\p)}\right)
\lesssim
\log\left(1/\epsilon\,\Vert F_n\Vert_{L_2(\p)}\right).
\]
Hence, since $\Vert F_n\Vert_{L_2(\p)}$ is bounded we obtain
\[
\int_0^{\delta_n}\sqrt{\log N_{[\,]}\left(\epsilon\Vert F_n\Vert_{L_2(\p)},\F_n, \Vert \cdot\Vert_{L_2(\p)}\right)}\,\mathrm{d}\epsilon
\lesssim
\delta_n + \int_0^{\delta_n}\sqrt{\log(1/\epsilon)}\,\mathrm{d}\epsilon
\to0.
\]
Finally, as in~\eqref{eqn:decompose}, for any $s\in[-1,1]$,
\[
\p f_{n,s}
=
b^{-1/2}
\left\{
H^{uc}(x-bs)-H(x)
\right\}
\to0.
\]
Furthermore, for $0\leq s\leq t$,
\[
\p f_{n,s}f_{n,t}
=
b^{-1}
\int_{x-bs}^x\frac{f(u)(1-G(u))}{\Phi^2(u)}\,\mathrm{d}u
=
b^{-1}
\int_{x-bs}^x\frac{\lambda(u)}{1-H(u)}\,\mathrm{d}u
\to
\frac{\lambda(x)}{1-H(x)}s.
\]
Similarly, for $t\leq s\leq 0$, $\p f_{n,s}f_{n,t}\to -s\lambda(x)/(1-H(x))$,
whereas $\p f_{n,s}f_{n,t}=0$, for $st<0$.
It follows that
\begin{equation}
\label{def:limiting cov}
\p f_{n,s}f_{n,t}-\p f_{n,s}\p f_{n,t}
\to
\begin{cases}
\dfrac{\lambda(x)}{1-H(x)}(|s|\wedge|t|) & \text{, if }st\geq 0;\\
0& \text{, if }st<0.
\end{cases}
\end{equation}
Consequently, according to~Theorem~2.11.23 in \citet{VW96}, the sequence of stochastic processes
$\sqrt{n}(\p_n-\p)f_{n,y}$ converges in distribution to a tight Gaussian process $\mathbb{G}$ with mean zero
and covariance given on the right hand side of~\eqref{def:limiting cov}.
Note that this is the covariance function of $\sqrt{\lambda(x)/[1-H(x)]}W$, where $W$ is a two sided Brownian motion.
We conclude that
\[
\begin{split}
&
n^{2/5}\int_{x-b}^{x+b} k_b(x-u)\,\mathrm{d}(\Lambda_n-\Lambda)(u)\\
&\quad=
\frac{1}{\sqrt{bn^{1/5}}} \int_{-1}^1 k(y)\,\mathrm{d}\hat{W}_n(y)\\
&\quad\stackrel{d}{\to}
\left(\frac{\lambda(x)}{c(1-H(x))}\right)^{1/2}
\int_{-1}^1 k(y)\,\mathrm{d}W(y)
\stackrel{d}{=}
N\left(0,\frac{\lambda(x)}{c\,(1-H(x))}\int_{-1}^1 k^2(y)\,\mathrm{d}y\right).
\end{split}
\]
This proves the first part of the theorem.

The optimal $c$ is then obtained by minimizing
\[
\mathrm{AMSE}(\tilde{\lambda}^{SG} ,c)=\frac{1}{4}\,c^4\,\lambda''(x)^2\, \left(\int u^2\,k(u)\,\mathrm{d}u\right)^2+\frac{\lambda(x)}{c\,(1-H(x))}\,\int k^2(u)\,\mathrm{d}u
\]
with respect to $c$.
\end{proof}
This result is in line with Theorem 11.8 in~\citet{GJ14} on the asymptotic distribution of the SMLE under the same model,
which highlights the fact that even after applying a smoothing technique the MLE and the Grenander-type estimator remain asymptotically equivalent.

Standard kernel density estimators lead to inconsistency problems at the boundary.
In order to prevent those, different methods of boundary corrections can be used.
Here we consider boundary kernels and one possibility is to construct linear combinations of $k(u)$ and $uk(u)$ with coefficients depending on the value near the boundary (see~\citet{DGL13}; \citet{ZK98}).
To be precise, we define the smoothed Grenander-type estimator $\hat{\lambda}_n^{SG}$ by
\begin{equation}
\label{eqn:sg}
\hat{\lambda}_n^{SG}(x)=\int_{(x-b)\vee 0}^{(x+b)\wedge\tau_H} k_b^{(x)}(x-u)\,\tilde{\lambda}_n(u)\,\mathrm{d}u=\int_{(x-b)\vee 0}^{(x+b)\wedge\tau_H} k_b^{(x)}(x-u)\,\mathrm{d}\tilde{\Lambda}_n(u),\quad x\in[0,\tau_H],
\end{equation}
with $k_b^{(x)}(u)$ denoting the rescaled kernel $b^{-1}k^{(x)}(u/b)$ and
\begin{equation}
\label{eqn:bound-kernel}
k^{(x)}(u)=
\begin{cases}
\phi(\frac{x}{b})k(u)+\psi(\frac{x}{b})uk(u) & x\in[0,b],\\
k(u) & x\in(b,\tau_H-b)\\
\phi(\frac{\tau_H-x}{b})k(u)-\psi(\frac{\tau_H-x}{b})uk(u) & x\in[\tau_H-b,\tau_H],
\end{cases}
\end{equation}
where $k(u)$ is a standard kernel satisfying~\eqref{def:kernel}.
For $s\in[-1,1]$, the coefficients $\phi(s)$, $\psi(s)$ are determined by
\begin{equation}
\label{eqn:coefficient}
\begin{aligned}
&\phi(s)\int_{-1}^s k(u)\,\mathrm{d}u+\psi(s)\int_{-1}^s uk(u)\,\mathrm{d}u=1,\\
&\phi(s)\int_{-1}^s uk(u)\,\mathrm{d}u+\psi(s)\int_{-1}^s u^2k(u)\,\mathrm{d}u=0.
\end{aligned}
\end{equation}
Note that $\phi$ and $\psi$ are not only well defined, but they are also continuously differentiable if the kernel $k$ is assumed to be continuous (see \citet{DGL13}).
Furthermore, it can be easily seen that, for each $x\in[0,b]$, equations~\eqref{eqn:coefficient} lead to
\begin{equation}
\label{eqn:integral}
\int_{-1}^{x/b}k^{(x)}(u)\,\mathrm{d}u=1\qquad\text{and}\qquad\int_{-1}^{x/b}uk^{(x)}(u)\,\mathrm{d}u=0.
\end{equation}
In this case, we obtain a stronger uniform consistency result which is stated in the next theorem.
\begin{theo}
\label{theo:boundcorr}
Let $\hat{\lambda}_n^{SG}$ be defined by~\eqref{eqn:sg} and suppose that $\lambda$ is nondecreasing and uniformly continuous.
Assume that $k$ satisfies~\eqref{def:kernel} and is differentiable with a uniformly bounded derivative
and that $bn^{\alpha}\to c\in(0,\infty)$.
If $0<\alpha<1/2$, then for any $0<M<\tau_H$,
\[
\sup_{x\in[0,M]}
\left|\hat{\lambda}_n^{SG}(x)-\lambda(x)\right|
\to
0
\]
in probability.
\end{theo}
\begin{proof}
For $x\in[0,M]$, write
\[
\hat{\lambda}_n^{SG}(x)-\lambda(x)
=
\left(\hat{\lambda}_n^{SG}(x)-\lambda_n^s(x)\right)
+
\Big(\lambda_n^{s}(x)-\lambda(x)\Big),
\]
where
\[
\lambda^s_n(x)=\int_{(x-b)\vee 0}^{(x+b)\wedge\tau_H} k_b^{(x)}(x-u)\,\lambda(u)\,\mathrm{d}u.
\]
We have to distinguish between two cases.
First, we consider the case $x\in[0,b]$.
By means of~\eqref{eqn:integral} and the fact that $\lambda$ is uniformly continuous, a change of variable yields
\begin{equation}
\label{eqn:cons2}
\sup_{x\in[0,b]}|\lambda^{s}_n(x)-\lambda(x)|
\leq
\sup_{x\in[0,b]}
\int_{-1}^{x/b} k^{(x)}(y)
\big|\lambda(x-by)-\lambda(x)\big|\,\mathrm{d}y
\to0.
\end{equation}
On the other hand, integration by parts and a change of variable give
\begin{equation}
\label{eqn:partcons2}
\begin{split}
\hat{\lambda}^{SG}_n(x)-\lambda^s_n(x)
&=
\int_{0}^{x+b}  k_b^{(x)}(x-u)\,\mathrm{d}\big(\tilde{\Lambda}_n-\Lambda\big)(u)\\
&=
-\int_0^{x+b}\frac{\partial}{\partial u}k_b^{(x)}(x-u)
\big(\tilde{\Lambda}_n(u)-\Lambda(u)\big)\,\mathrm{d}u\\
&=
\frac{1}{b}
\int_{-1}^{x/b}
\frac{\partial}{\partial y}k_b^{(x)}(y)
\big(\tilde{\Lambda}_n(x-by)-\Lambda(x-by)\big)\,\mathrm{d}y.
\end{split}
\end{equation}
Consequently, since for $n$ sufficiently large $x+b\leq M$, we obtain
\[
\sup_{x\in[0,b]}
\left|\hat{\lambda}^{SG}_n(x)-\lambda^s_n(x)\right|
\lesssim
\frac{1}{b}
\sup_{u\in[0,M]}
\left|
\tilde{\Lambda}_n(u)-\Lambda(u)
\right|
=O_p(n^{-1/2+\alpha}),
\]
because of~\eqref{eqn:chbound} and the boundedness of the coefficients $\phi$, $\psi$ and of $k(u)$ and $k'(u)$.
Together with~\eqref{eqn:cons2} and since $0<\alpha<1/2$, this proves that,
\[
\sup_{x\in[0,b]}
\left|\hat{\lambda}^{SG}_n(x)-\lambda(x)\right|
=o_P(1).
\]
When $x\in(b,M]$, for sufficiently large $n$, we have $0<x-b<x+b<\tau_H$, so that
by a change of variable and uniform continuity of $\lambda$,
it follows that
\begin{equation}
\label{eqn:cons1}
\sup_{x\in(b,M]}
\left|\lambda^{s}_n(x)-\lambda(x)\right|
\leq
\int_{-1}^{1} k(y)
|\lambda(x-by)-\lambda(x)|\,\mathrm{d}y
\to0.
\end{equation}
Furthermore,
\[
\hat{\lambda}^{SG}_n(x)
=
\int_{x-b}^{x+b} k_b(x-u)\,\tilde{\lambda}_n(u)\,\mathrm{d}u,
\]
so that, arguing as in~\eqref{eqn:partcons2}, we find that
\[
\sup_{x\in(b,M]}
\left|\hat{\lambda}^{SG}_n(x)-\lambda^s_n(x)\right|
=O_p(n^{-1/2+\alpha}),
\]
which, together with~\eqref{eqn:cons1}, proves that
\[
\sup_{x\in(b,M]}\left|\hat{\lambda}^{SG}_n(x)-\lambda(x)\right|
=o_P(1).
\]
This proves the theorem.
\end{proof}
The previous result illustrates that even if we use boundary kernels, we can not avoid inconsistency problems at the end point of the support.
Although a bit surprising, this is to be expected because we can only control the distance between the Nelson-Aalen estimator and the cumulative hazard on intervals that stay away from the right boundary (see~\eqref{eqn:chbound}).
\begin{figure}
\includegraphics[width=.6\textwidth]{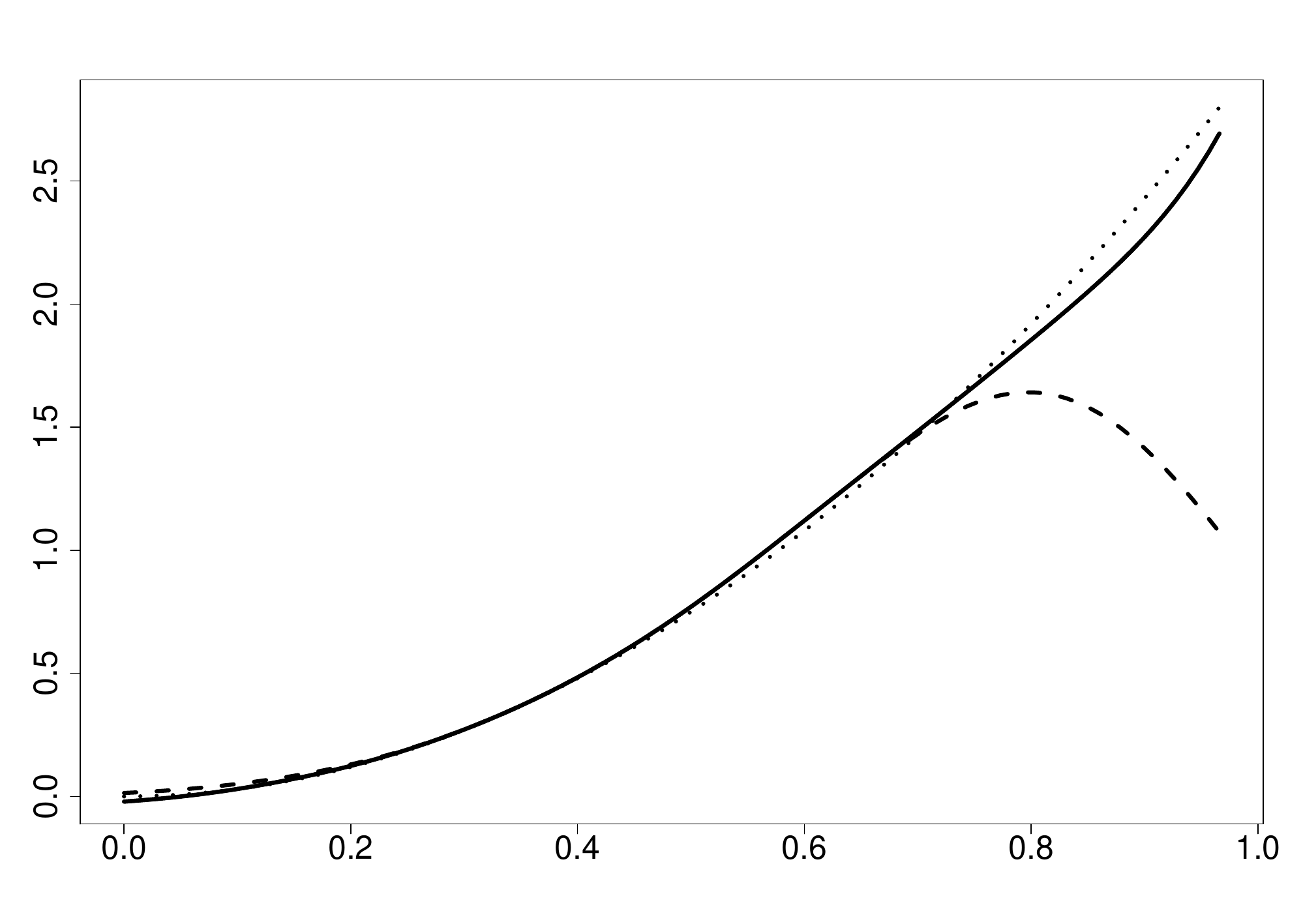}
\caption{The kernel smoothed versions (standard-dashed line and boundary corrected-solid line) of the hazard rate (dotted line).}
\label{fig:bound}
\end{figure}
Figure~\ref{fig:bound} illustrates that  boundary corrections improve the performance of the smooth estimator constructed with the standard kernel.
\section{Smoothed Grenander-type estimator of a monotone density}\label{sec:dens}
This section is devoted to the smoothed Grenander-type estimator $\tilde{f}_n^{SG}$ of an increasing density $f$. Let $k$ be a kernel function satisfying~\eqref{def:kernel}. For a fixed $x\in[0,\tau_H]$, $\tilde{f}_n^{SG}$ is defined by
\begin{equation}
\label{eqn:SGd}
\tilde{f}_n^{SG}(x)=\int_{(x-b)\vee 0}^{(x+b)\wedge\tau_H} k_b(x-u)\,\tilde{f}_n(u)\,\mathrm{d}u=\int k_b(x-u)\,\mathrm{d}\tilde{F}_n(u).
\end{equation}
We also consider the boundary corrected version $\hat{f}_n^{SG}$ of the smoothed Grenander-type estimator, defined by
\begin{equation}
\label{eqn:sgd}
\hat{f}_n^{SG}(x)=\int_{(x-b)\vee 0}^{(x+b)\wedge\tau_H} k_b^{(x)}(x-u)\,\tilde{f}_n(u)\,\mathrm{d}u=\int_{(x-b)\vee 0}^{(x+b)\wedge\tau_H} k_b^{(x)}(x-u)\,\mathrm{d}\tilde{F}_n(u),\quad x\in[0,\tau_H],
\end{equation}
with $k_b^{(x)}(u)$ as in~\eqref{eqn:bound-kernel}.
The following results can be proved in exactly the same way as Theorem~\ref{theo:cons} and Theorem~\ref{theo:boundcorr}.
\begin{theo}
\label{theo:cons-dens}
Let $k$ be a kernel function satisfying~\eqref{def:kernel} and let
$\tilde{f}^{SG}$ be the smoothed Grenander-type estimator defined in~\eqref{eqn:SGd}. Suppose that the density function $f$ is nondecreasing and continuous. Then for each $0<\epsilon<\tau_H$, it holds
\[
\sup_{x\in[\epsilon,\tau_H-\epsilon]}|\tilde{f}_n^{SG}(x)-f(x)|\to 0
\]
with probability one.
\end{theo}
\begin{theo}
\label{theo:boundcorr-dens}
Let $\hat{f}_n^{SG}$ be defined by~\eqref{eqn:sg} and suppose that $f$ is nondecreasing and uniformly continuous.
Assume that $k$ satisfies~\eqref{def:kernel} and is differentiable with uniformly bounded derivatives
and that $b\,n^{\alpha}\to c\in(0,\infty)$.
If $0<\alpha<1/2$, then for any $0<M<\tau_H$,
\[
\sup_{x\in[0,M]}
\left|\hat{f}_n^{SG}(x)-f(x)\right|
\to
0
\]
in probability.
\end{theo}
\begin{figure}
\includegraphics[width=.6\textwidth]{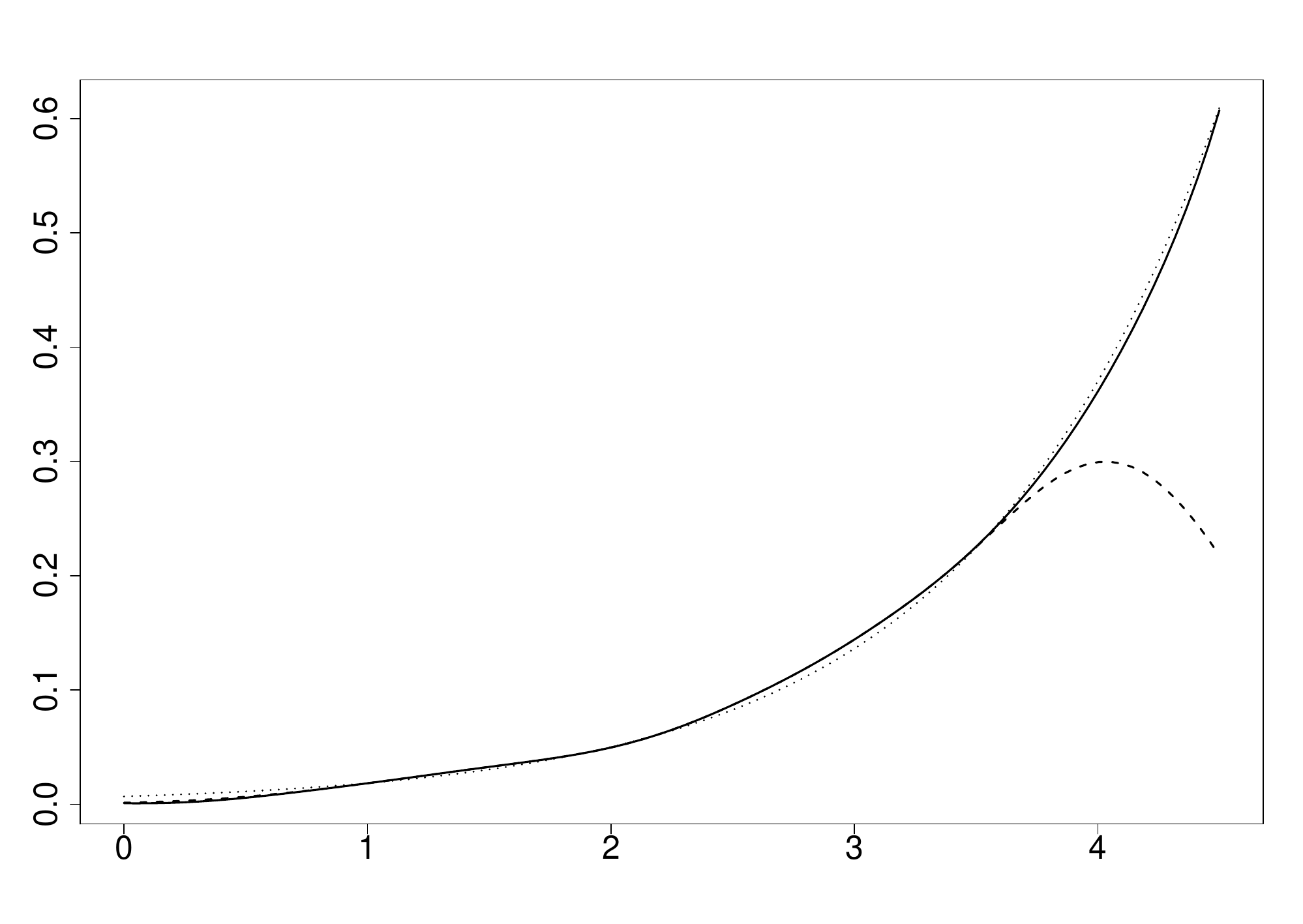}
\caption{The kernel smoothed versions (standard-dashed line and boundary corrected-solid line) of the density function (dotted line).}
\label{fig:bound-dens}
\end{figure}
Figure~\ref{fig:bound-dens} shows the smooth isotonic estimators of an increasing density {up to the $90\%$ quantile of $H$} for a sample of size $n=500$.
We choose {
\begin{equation}
\label{eqn:f}
f(x)=(e^5-1)^{-1}e^x\mathds{1}_{[0,5]}(x)
\end{equation}
and 
\begin{equation}
\label{eqn:g}
g(x)=(e^{5/2}-1)^{-1}e^{x/2}\mathds{1}_{[0,5]}(x)
\end{equation}
for the censoring times.}
The bandwidth used is $b=c_{opt}\,n^{-1/5}$, where $c_{opt}=4$ is the asymptotically MSE-optimal constant (see~\eqref{eqn:c_opt-dens})
corresponding to $x_0=2.5$.

In order to derive the asymptotic normality of the smoothed Grenander-type estimator $\tilde{f}_n^{SG}$ we first provide a Kiefer-Wolfowitz type of result for the Kaplan-Meier estimator.
\begin{lemma}
\label{le:KW}
Let $0<M<\tau_H$. Let $f$ be a nondecreasing and  continuously differentiable density such that $f$ and $f'$ are strictly positive.
Then we have
\[
\sup_{x\in[0,M]}|\tilde{F}_n(x)-F_n(x)|=O_P\left(\frac{\log n}{n}\right)^{2/3},
\]
where $F_n$ is the Kaplan-Meier estimator and $\tilde{F}_n$ is the greatest convex minorant of $F_n$.
\end{lemma}
\begin{proof}
We consider $f$ on the interval $[0,M]$ and apply Theorem~2.2 in~\citet{DL14}.
The density $f$ satisfies condition (A1) of this theorem with $[a,b]=[0,M]$.
Condition (2) of Theorem~2.2 in~\citet{DL14} is provided by
the strong approximation~\eqref{eqn:approx}, with $L$ defined in~\eqref{eqn:L}, $\gamma_n=O(n^{-1}\log n)^{2/3}$, and
\[
B(t)=\left(1-F(L^{-1}(t))\right)W(t),
\quad
t\in[L(0),L(M)]
\]
where $W$ is a Brownian motion.
It remains to show that $B$ satisfies conditions (A2)-(A3) of Theorem~2.2 in~\citet{DL14} with $\tau=1$.
In order to check these conditions for the process $B$, let
$x\in[L(0),L(M)]=[0,L(M)]$, $u\in(0,1]$ and $v>0$.
Then
\begin{equation}
\label{eq:bound A2}
\begin{split}
&
\p\left(
\sup_{|x-y|\leq u} |B(x)-B(y)|>v\right)\\
&\leq
\p\left(\sup_{|x-y|\leq u,\, y\in[0,L(M)]} |B(x)-B(y)|>v\right)\\
&\leq
2\p\left( \sup_{|x-y|\leq u} |W(x)-W(y)|>\frac{v}{3}\right)\\
&\qquad+
\p\left( \sup_{|x-y|\leq u,\, y\in[0,L(M)]} |F(L^{-1}(x))-F(L^{-1}(y)) |\,|W(x)|>\frac{v}{3}\right).
\end{split}
\end{equation}
Note that from the proof of Corollary 3.1 in~\citet{DL14} it follows that $W$ satisfies condition (A2) in~\citet{DL14}.
This means that there exist $K_1,K_2>0$, such that the first probability on the right hand side of~\eqref{eq:bound A2} is bounded by
$K_1\exp(-K_2v^2u^{-1})$.
Furthermore, since $u\in(0,1]$ and
\[
|F(L^{-1}(x))-F(L^{-1}(y)) |\leq \frac{\sup_{u\in[0,M]}f(u)}{\inf_{u\in[0,M]}L'(u)}\,|x-y|\leq \frac{\sup_{u\in[0,M]}f(u)}{\inf_{u\in[0,M]}\lambda(u)}\,|x-y|,
\]
the second probability on the right hand side of~\eqref{eq:bound A2} is bounded by
\[
\p\left(|W(x)|>\frac{v}{K_3\sqrt{u}}\right)
\leq
\p\left(\sup_{t\in[0,L(M)]}|W(t)|>\frac{v}{K_3\sqrt{u}}\right),
\]
for some $K_3>0$.
Hence, by the maximal inequality for Brownian motion, we conclude that
there exist $K_1',K_2'>0$ such that
\[
\p\left(
\sup_{|x-y|\leq u} |B(x)-B(y)|>v\right)
\leq
K_1'\exp\left(-K_2'v^2u^{-1}\right)
\]
which proves condition (A2) in~\citet{DL14}.

Let us now consider (A3).
For all $x\in[0,L(M)]$, $u\in(0,1]$, and $v>0$, we obtain
\begin{equation}
\label{eq:bound A3}
\begin{split}
&
\p\left( \sup_{u\leq z\leq x}
\left\{B(x-z)-B(x)-vz^2\right\}>0\right)\\
&\leq
\p\left( \sup_{u\leq z\leq x}
\left\{W(x-z)-W(x)-\frac{vz^2}{3}\right\}>0\right)\\
&\quad+
\p\left( \sup_{u\leq z\leq x} \left\{F(L^{-1}(x-z))[W(x)-W(x-z)]-\frac{vz^2}{3}\right\}>0\right)\\
&\qquad+
\p\left(\sup_{u\leq z\leq x}
\left\{
\left(F(L^{-1}(x))-F(L^{-1}(x-z))\right)W(x)-\frac{vz^2}{3}
\right\}>0\right).
\end{split}
\end{equation}
Again, from the proof of Corollary 3.1 in~\citet{DL14} it follows that $W$ satisfies condition (A3) in~\citet{DL14},
which means that there exist $K_1,K_2>0$, such that the first probability on the right hand side of~\eqref{eq:bound A3} is bounded by
$K_1\exp(-K_2v^2u^{3})$.
We establish the same bound for the remaining tow probabilities.
By the time reversal of the Brownian motion, the process $W'(z)=W(x)-W(x-z)$ is also a Brownian motion on the interval $[u,x]$.
Then, using the change of variable $u/z=t$ and the fact that $\widetilde{W}(t)=tu^{-1/2}W'(u/t)$, for $t>0$, is again a Brownian motion,
the second probability on the right hand side of~\eqref{eq:bound A3} is bounded by
\begin{equation}
\label{eq:bound A3 term2}
\p\left(\sup_{t\in(0,1]} \frac{t}{\sqrt{u}}\left|W'\left(\frac{u}{t}\right)\right|>\frac{vu^{3/2}}{3t}\right)
=
\p
\left(
\sup_{t\in[0,1]}
|\widetilde{W}(t)|>\frac{vu^{3/2}}{3}
\right).
\end{equation}
Finally,
\[
\sup_{u\leq z\leq x} \left|\frac{F(L^{-1}(x))-F(L^{-1}(x-z))}{z}\right|\leq \frac{\sup_{u\in[0,M]}f(u)}{\inf_{u\in[0,M]}\lambda(u)},
\]
so that the third probability on the right hand side of~\eqref{eq:bound A3} is bounded by
\begin{equation}
\label{eq:bound A3 term3}
\begin{split}
&
\p\left(\sup_{u\leq z\leq x} \left|\frac{F(L^{-1}(x))-F(L^{-1}(x-z))}{z}\right|\,|W(x)|>\frac{vu}{3}\right)\\
&\qquad\leq
\p\left(|W(x)|>\frac{vu^{3/2}}{K_3}\right)
\leq
\p\left(\sup_{t\in[0,L(M)]}|W(t)|>\frac{vu^{3/2}}{K_3}\right),
\end{split}
\end{equation}
for some $K_3>0$.
By applying the maximal inequality for Brownian motion to the right hand sides of~\eqref{eq:bound A3 term2} and~\eqref{eq:bound A3 term3},
we conclude that there exist $K_1',K_2'>0$, such that
\[
\p\left( \sup_{u\leq z\leq x} \{B(x-z)-B(x)-vz^2\}>0\right)
\leq K'_1\exp(-K'_2v^2u^{3}),
\]
which proves condition (A3) in~\citet{DL14}.
\end{proof}
\begin{theo}
\label{theo:distr-dens}
Let $f$ be a nondecreasing and twice continuously differentiable density such that $f$ and $f'$ are strictly positive. Let $k$ satisfy~\eqref{def:kernel} and suppose that it is differentiable with a uniformly bounded derivative. If $bn^{1/5}\to c\in(0,\infty)$, then for each $x\in(0,\tau_h)$, it holds
\[
n^{2/5}\,\big(\,\tilde{f}_n^{SG}(x)-f(x)\,\big)\xrightarrow{d}N(\mu,\sigma^2),
\]
where
\begin{equation}
\label{eqn:mu-sigma-dens}
\mu=\frac{1}{2}\,c^2\,f''(x)\, \int u^2\,k(u)\,\mathrm{d}u\quad\text{ and }\quad \sigma^2= \frac{f(x)}{c\,(1-G(x))}\,\int k^2(u)\,\mathrm{d}u.
\end{equation}
For a fixed $x\in(0,\tau_h)$, the asymptotically MSE-optimal bandwidth $b$ for $\tilde{\lambda}^{SG}$ is given by
$c_{opt}(x)n^{-1/5}$, where
\begin{equation}
\label{eqn:c_opt-dens}
c_{opt}=\left\{f(x)\int k^2(u)\,\mathrm{d}u \right\}^{1/5}\left\{(1-G(x))f''(x)^2 \left(\int u^2\,k(u)\,\mathrm{d}u\right)^2\right\}^{-1/5}.
\end{equation}
\end{theo}
\begin{proof}
Fix $x\in(0,\tau_H)$. Then, for sufficiently large  $n$, we have $0<x-b< x+b\leq M<\tau_H.$ Following the proof of~\ref{theo:distr}, we write
\begin{equation}
\label{eqn:distr1-dens}
\tilde{f}^{SG}_n(x)=\int k_b(x-u)\,\mathrm{d}F(u)+\int k_b(x-u)\,\mathrm{d}(F_n-F)(u)+\int k_b(x-u)\,\mathrm{d}(\tilde{F}_n-F_n)(u).
\end{equation}
Again the first (deterministic) term on the right hand side of~\eqref{eqn:distr1-dens} gives us the asymptotic bias:
\[
n^{2/5}
\left\{
\int_{x-b}^{x+b} k_b(x-u)f(u)\,\mathrm{d}u-f(x)
\right\}
\to
\frac{1}{2}c^2f''(x)
\int_{-1}^1 y^2 k(y)\,\mathrm{d}y,
\]
and by the Kiefer-Wolfowitz type of result in Lemma~\ref{le:KW}, the last term on the right hand side of~\eqref{eqn:distr1-dens} converges to $0$ in probability. What remains is to prove that
\[
n^{2/5}\,\int k_b(x-u)\,\mathrm{d}(F_n-F)(u)\xrightarrow{d} N(0,\sigma^2),
\]
where $\sigma^2$ is defined in~\eqref{eqn:mu-sigma-dens}. We write
\[
n^{2/5}\int_{x-b}^{x+b} k_b(x-u)\,\mathrm{d}(F_n-F)(u)=\frac{1}{\sqrt{b\,n^{1/5}}} \int_{-1}^1 k(y)\,\mathrm{d}\hat{W}_n(y),
\]
where, for each $y\in[-1,1]$, we define
\begin{equation}
\label{eqn:distr2-dens}
\begin{split}
\hat{W}_n(y)
&=
\sqrt{\frac{n}{b}}
\Big\{F_n(x-by)-F_n(x)-F(x-by)+F(x)\Big\}\\
&=
\sqrt{\frac{n}{b}}
\Big\{F_n(x-by)-F(x-by)-n^{-1/2}(1-F(x-by))W\circ L(x-by)\Big\}\\
&\qquad-
\sqrt{\frac{n}{b}}
\left\{F_n(x)-F(x)-n^{-1/2}(1-F(x))W\circ L(x)\right\}\\
&\qquad+
\frac{1}{\sqrt{b}}
\Big(1-F(x)\Big)\Big\{W\circ L(x-by)-W\circ L(x)\Big\}\\
&\qquad+
\frac{1}{\sqrt{b}}\Big(F(x)-F(x-by)\Big)W\circ L(x-by).
\end{split}
\end{equation}
Using the strong approximation~\eqref{eqn:approx}, we obtain
\[
\begin{split}
&
\p\left(\sup_{u\in[0,M]}\sqrt{\frac{n}{b}}\left|F_n(u)-F(u)-n^{-1/2}(1-F(u))W\circ L(u)\right|>\epsilon\right)\\
&\qquad\leq
K_1\exp\left\{-K_2(\epsilon\sqrt{nb}-K_3\log n)\right\}\to 0,
\end{split}
\]
and it then follows that the first two terms on the right hand side of~\eqref{eqn:distr2-dens} converge to $0$ in probability uniformly in $y$.
For the last term, we get
\[
\begin{split}
&
\p\left(\sup_{y\in[-1,1]}
\left|\frac{1}{\sqrt{b}}\Big(F(x)-F(x-by)\Big)W\circ L(x-by)\right|>\epsilon\right)\\
&\qquad\leq
\p\left(\sqrt{b}\sup_{u\in[0,M]}f(u)\sup_{u\in [0,\Vert L\Vert_\infty]}|W(u)|>\epsilon\right)
\leq K_1\exp\left(-\frac{K_2\epsilon^2}{b}\right)\to 0.
\end{split}
\]
For the third term on the right hand side of~\eqref{eqn:distr2-dens}, note
that $y\mapsto b^{-1/2}(W\circ L(x-by)-W\circ L(x))$, for $y\in[-1,1]$, has the same distribution as the process
\begin{equation}
\label{eqn:process}
y\mapsto \widetilde{W}\left(\frac{L(x)-L(x-by)}{b}\right),
\quad
\text{for }y\in[-1,1],
\end{equation}
where $\widetilde{W}$ is a two-sided Brownian motion.
By uniform continuity of the two-sided Brownian motion on compact intervals,
the sequence of stochastic processes in~\eqref{eqn:process} converges to the process
$\{\widetilde{W}\left(L'(x)y\right):y\in[-1,1]\}$:
\[
\sup_{y\in[-1,1]}\left|\widetilde{W}\left(\frac{L(x)-L(x-by)}{b}\right)-\widetilde{W}\left(L'(x)y\right)\right|\xrightarrow{\p} 0.
\]
As a result
\[
\begin{split}
\frac{1}{\sqrt{bn^{1/5}}}
\int_{-1}^1 k(y)\,\mathrm{d}\hat{W}_n(y)
&\xrightarrow{d}
\sqrt{\frac{f(x)}{c(1-G(x))}}
\int_{-1}^1 k(y)\,\mathrm{d}\widetilde{W}(y)\\
&\sim
N\left(0,\frac{f(x)}{c(1-G(x))}\int_{-1}^1 k^2(y)\,\mathrm{d}y \right).
\end{split}
\]
The optimal $c$ is then obtained by minimizing
\[
\mathrm{AMSE}(\tilde{f}^{SG} ,c)=\frac{1}{4}c^4f''(x)^2
\left(\int u^2\,k(u)\,\mathrm{d}u\right)^2
+
\frac{f(x)}{c(1-G(x))}
\int k^2(u)\,\mathrm{d}u,
\]
with respect to $c$.
\end{proof}
\section{Pointwise confidence intervals}
\label{sec:conf-int}
In this section we construct pointwise confidence intervals for the hazard rate and the density based on the asymptotic distributions derived in
Theorem~\ref{theo:distr} and Theorem~\ref{theo:distr-dens} and compare them to confidence intervals constructed using Grenander-type estimators without smoothing.
According to Theorem~2.1 and Theorem~2.2 in~\citet{huang-wellner1995}, for a fixed $x_0\in(0,\tau_H)$,
\[
n^{1/3}\left|\frac{1-G(x_0)}{4f(x_0)f'(x_0)}\right|^{1/3}\left(\tilde{f}_n(x_0)-f(x_0)\right)
\xrightarrow{d}
\Z,
\]
and
\[
n^{1/3}\left|\frac{1-H(x_0)}{4\lambda(x_0)\lambda'(x_0)}\right|^{1/3}\left(\tilde{\lambda}_n(x_0)-\lambda(x_0)\right)
\xrightarrow{d}
\Z,
\]
where $W$ is a two-sided Brownian motion starting from zero
and $\Z=\argmin_{t\in\R}\left\{W(t)+t^2\right\}$.
This yields $100(1-\alpha)\%$-confidence intervals for $f(x_0)$ and $\lambda(x_0)$ of the following form
\[
C^1_{n,\alpha}
=
\tilde{f}_n(x_0)\pm n^{-1/3}\hat{c}_{n,1}(x_0)q(\Z,1-\alpha/2),
\]
and
\[
C^2_{n,\alpha}
=
\tilde{\lambda}_n(x_0)\pm n^{-1/3}\hat{c}_{n,2}(x_0)q(\Z,1-\alpha/2),
\]
where $q(\Z,1-\alpha/2)$ is the $(1-\alpha/2)$ quantile of the distribution $\Z$ and
\[
\hat{c}_{n,1}(x_0)
=
\left|\frac{4\tilde{f}_n(x_0)\tilde{f}_n'(x_0)}{1-G_n(x_0)}\right|^{1/3},
\qquad
\hat{c}_{n,2}(x_0)
=
\left|\frac{4\tilde{\lambda}_n(x_0)\tilde{\lambda}'_n(x_0)}{1-H_n(x_0)}\right|^{1/3}.
\]
Here, $H_n$ is the empirical distribution function of $T$ and in order to avoid the denominator taking the value zero,
instead of the natural estimator of $G$, we consider a slightly different version as in~\citet{MP87}:
\begin{equation}
\label{eqn:G_n}
G_n(t)
=
\begin{cases}
0 & \text{ if } 0\leq t< T_{(1)},\\
\displaystyle{1-\prod_{i=1}^{k-1}\left(\frac{n-i+1}{n-i+2} \right)^{1-\Delta_i}} & \text{ if } T_{(k-1)}\leq t< T_{(k)},\quad k=2,\dots,n,\\
\displaystyle{1-\prod_{i=1}^{n}\left(\frac{n-i+1}{n-i+2} \right)^{1-\Delta_i}} & \text{ if } t\geq T_{(n)}.
\end{cases}
\end{equation}
Furthermore, as an estimate for $\tilde{f}_n'(x_0)$ we choose $(\tilde{f}_n(\tau_m)-\tilde{f}_n(\tau_{m-1}))/(\tau_m-\tau_{m-1})$,
where $\tau_{m-1}$ and $\tau_m$ are two succeeding points of jump of $\tilde{f}_n$ such that $x_0\in(\tau_{m-1},\tau_m]$,
and~$\tilde{\lambda}_n'(x_0)$ is estimated similarly.
The quantiles of the distribution $\Z$ have been computed in~\citet{GW01} and we will use $q(\Z,0.975)=0.998181$.

The pointwise confidence intervals based on the smoothed Grenander-type estimators are constructed from
Theorem~\ref{theo:distr} and Theorem~\ref{theo:distr-dens}.
We find
\[
\widetilde{C}^1_{n,\alpha}
=
\tilde{f}^{SG}_n(x_0)\pm n^{-2/5}
\left(
\hat{\sigma}_{n,1}(x_0)q_{1-\alpha/2}+\hat{\mu}_{n,1}(x_0)
\right),
\]
and
\[
\widetilde{C}^2_{n,\alpha}
=
\tilde{\lambda}^{SG}_n(x_0)\pm
n^{-2/5}(\hat{\sigma}_{n,2}(x_0)q_{1-\alpha/2}+\hat{\mu}_{n,2}(x_0)),
\]
where $q_{1-\alpha/2}$ is the $(1-\alpha/2)$ quantile of the standard normal distribution.
The estimators $\hat{\sigma}_{n,1}(x_0)$ and $\hat{\mu}_{n,1}(x_0)$ are obtained by plugging $\tilde{f}^{SG}_n$ and its second derivative for $f$ and $f''$, respectively,
and $G_n$ and $c_{opt}(x_0)$ for $G$ and $c$, respectively, in~\eqref{eqn:mu-sigma-dens}, and similarly $\hat{\sigma}_{n,2}(x_0)$ and $\hat{\mu}_{n,2}(x_0)$
are obtained from~\eqref{eqn:mu-sigma}.
Estimating the bias seems to be a hard problem because it depends on the second derivative of the function of interest.
As discussed, for example in~\citet{Hall92}, one can estimate the bias by using a bandwidth of a different order for estimating the second derivative or one
can use undersmoothing (in that case the bias is zero and we do not need to estimate the second derivative).
We tried both methods and it seems that undersmoothing performs better, which is in line with other results available in the literature (see for instance,~\citet{Hall92};~\citet{GJ15};~\citet{CHT06}).

When estimating the hazard rate,
we choose a Weibull distribution with shape parameter $3$ and scale parameter $1$ for the event times and the uniform distribution on~$(0,1.3)$ for the censoring times.
Confidence intervals are calculated at the point $x_0=0.5$ using 1000 sets of data and the bandwidth in the case of undersmoothing is $b=c_{opt}(x_0)n^{-1/4}$,
where $c_{opt}(x_0)=1.2$.
In the case of bias estimation we use $b=c_{opt}(x_0)n^{-5/17}$ to estimate the hazard and $b_1=c_{opt}(x_0) n^{-1/17}$ to estimate its second derivative (as suggested in~\citet{Hall92}).

Table~\ref{tab:1} shows the performance, for various sample sizes, of the confidence intervals based on the asymptotic distribution (AD) of the Grenander-type estimator and of the smoothed Grenander estimator (for both undersmoothing and bias estimation).
\begin{table}
\begin{tabular}{ccccccc}
 \toprule
      & \multicolumn{2}{c}{Grenander}      &   \multicolumn{2}{c}{SG-undersmoothing} &   \multicolumn{2}{c}{SG-bias estimation}  \\
n     & AL    & CP    & AL    & CP    & AL    & CP     \\
100   & 0.930  & 0.840  & 0.648 & 0.912 & 0.689 & 0.955  \\
500   & 0.560  & 0.848 & 0.366 & 0.948 & 0.383 & 0.975  \\
1000  & 0.447 & 0.847 & 0.283 & 0.954 & 0.295 & 0.977  \\
5000  & 0.255 & 0.841  & 0.155 & 0.953 & 0.157 & 0.978  \\
\bottomrule
\end{tabular}
\caption{The average length (AL) and the coverage probabilities (CP) for $95\%$ pointwise confidence intervals of the hazard rate at the point $x_0=0.5$ based on the asymptotic distribution. }
\label{tab:1}
\end{table}
The poor performance of the Grenander-type estimator seems to be related to the crude estimate of the derivative of the hazard with the slope of the correspondent segment.
On the other hand, it is obvious that smoothing leads to significantly better results in terms of both average length and coverage probabilities.
As expected, when using undersmoothing, as the sample size increases we get shorter confidence intervals and coverage probabilities that are closer to the nominal level of $95\%$.
By estimating the bias, we obtain coverage probabilities that are  higher than $95\%$,
because the confidence intervals are bigger compared to the average length when using undersmoothing.
\begin{figure}
\centering
\subfloat[][]
{\includegraphics[width=.48\textwidth]{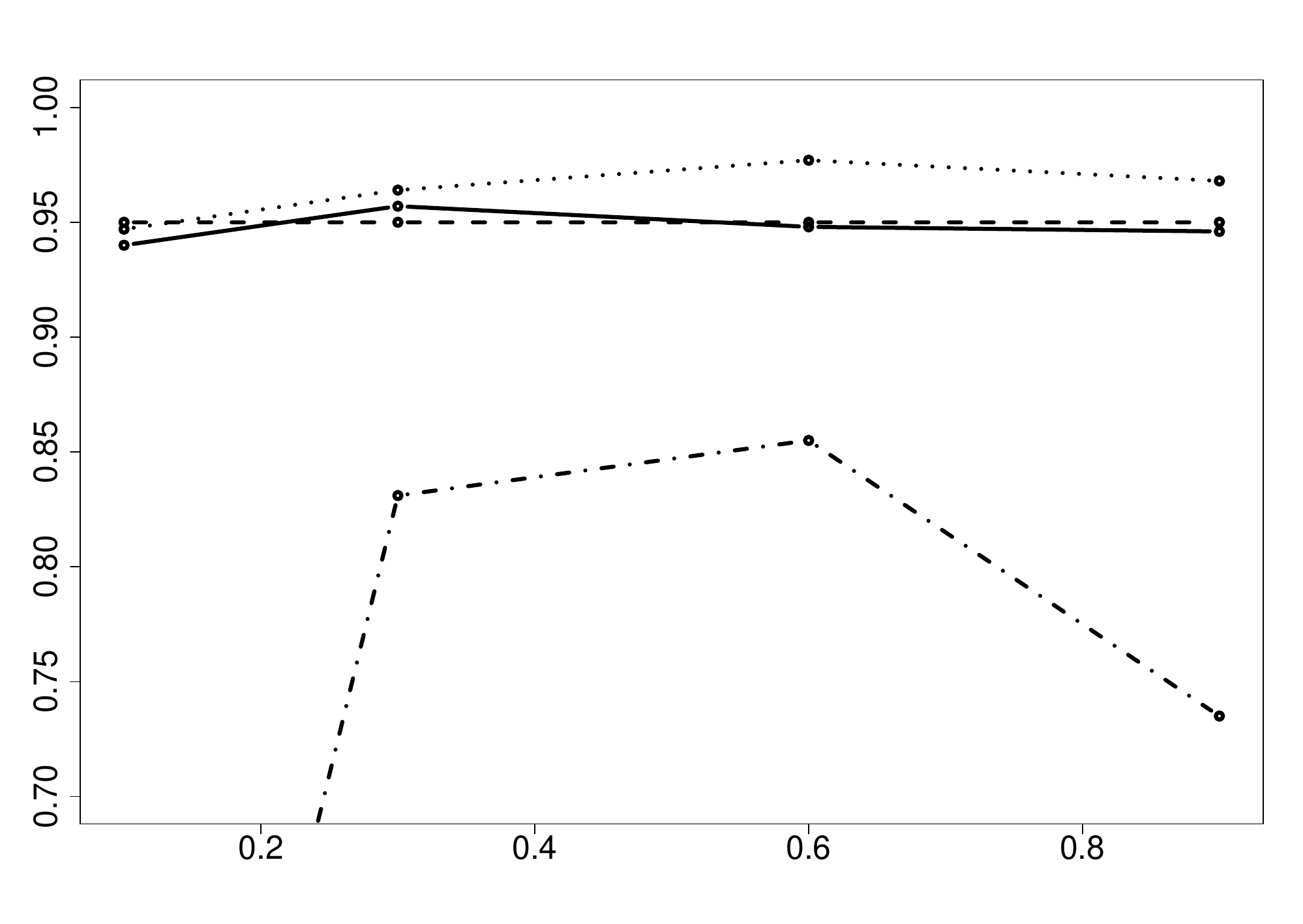}} \quad
\subfloat[][]
{\includegraphics[width=.48\textwidth]{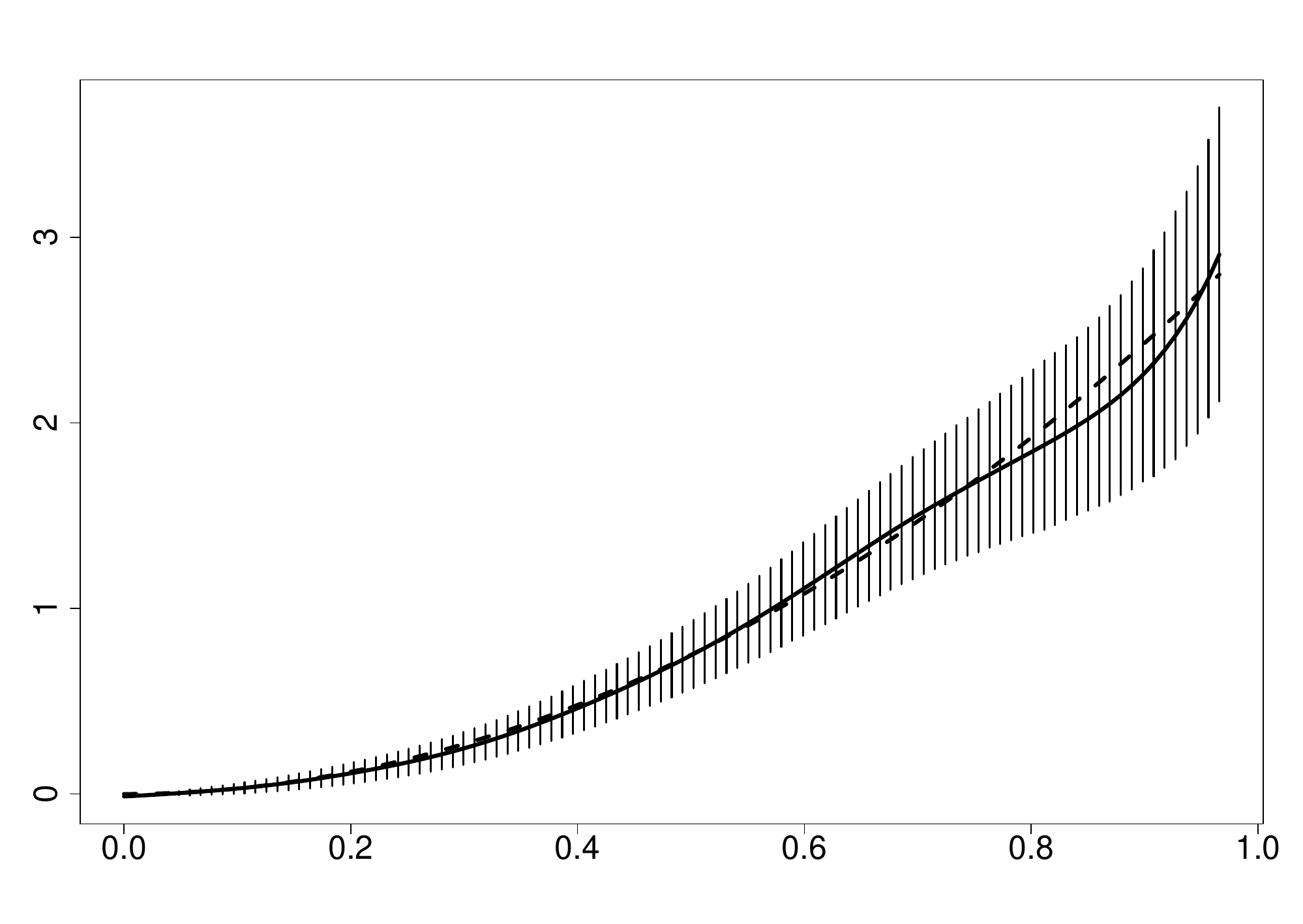}}
\caption{Left panel: Actual coverage of confidence intervals with nominal coverage $95\%$ for the hazard rate based on the asymptotic distribution. Dashed line-nominal level; dotdashed line-Grenander-type; solid line-SG undersmoothing; dotted line-SG bias estimation. Right panel: $95\%$ confidence intervals based on the asymptotic distribution for the hazard rate using undersmoothing.}
\label{fig:subfig4}
\end{figure}
Another way to compare the performance of the different methods is to take
a fixed sample size $n=500$ and different points of the support of the hazard function.
Figure~\ref{fig:subfig4} shows that confidence intervals based on undersmoothing behave well also at the boundaries in terms of coverage probabilities,
but the length increases as we move to the left end point of the support.
In order to maintain good visibility of the performance of the smooth estimators, we left out the poor performance of the Grenander estimator at point $x=0.1$.

{It is also of interest to compare our confidence intervals with other competing methods, in particular with those obtained via the inversion of the likelihood ratio statistic proposed in \citet{MB08}. We consider their setting of simulation A.2 (heavy cencoring case), where the event times come from a Weibull distribution with shape parameter $2$ and scale $\sqrt{2}$ and the censoring times are uniform in $(0,1.5)$. Again, we recorded the average length of nominal $95\%$ confidence intervals for the hazard rate at the point $x_0=\sqrt{2\log 2}$ obtained using undersmoothing and their coverage probabilities. The results for various sample sizes and choices of the constant $c$ in the definition of the bandwidth are displayed in Table~\ref{tab:3}. These observations show that the performance of the confidence intervals strongly depends on the choice of $c$. Most of the time the confidence intervals are shorter compared to those using likelihood ratio and for $c=1$ or $c=1.2$ our method produces also better coverage probabilities. On the other hand, if $c$ is too small, e.g. $c=0.8$, the confidence intervals become quite conservative (or anticonservative for large $c$). Of course the likelihood ratio method has the advantage of not requiring estimation of nuisance parameters or choosing the bandwidth but our results confirm that with the right choice of the bandwidth smooth estimation performs much better. The importance of the choice of the smoothing parameter is well-known and different methods has been proposed in the literature to find the optimal one (see for example \citet{CHT06} and \citet{GCM96}). However, it is beyond the scope of this paper to investigate methods of bandwidth selection.
\begin{table}
\begin{tabular}{ccccccccc}
 \toprule
     & \multicolumn{2}{c}{LR}  & \multicolumn{2}{c}{c=0.8}      &   \multicolumn{2}{c}{c=1} &   \multicolumn{2}{c}{c=1.2}  \\
n   & AL & CP  & AL    & CP    & AL    & CP    & AL    & CP     \\
50   & 3.110 & 0.911 & 2.951 & 0.962 & 2.700 & 0.946 & 2.463 & 0.925  \\
100  & 2.408 & 0.917 & 2.213 & 0.964 & 1.991 & 0.951 & 1.820 & 0.935  \\
200  & 1.684 & 0.929 & 1.680 & 0.972 & 1.512 & 0.947 & 1.392 & 0.930  \\
500  & 1.073 & 0.932 & 1.202 & 0.970 & 1.070 & 0.952 & 0.981 & 0.927  \\
1000 & 0.782 & 0.936 & 0.935 & 0.975 & 0.836 & 0.958 & 0.764 & 0.936  \\
1500 & 0.653 & 0.941 & 0.809 & 0.982 & 0.720 & 0.965 & 0.663 & 0.944  \\
\bottomrule
\end{tabular}
\caption{The average length (AL) and the coverage probabilities (CP) for $95\%$ pointwise confidence intervals of the hazard rate at the point $x_0=\sqrt{2\log 2}$ using likelihood ratio (LR) and undersmoothing with various choices of $c$ .}
\label{tab:3}
\end{table}
}
Finally, we consider estimation of the density.
We simulate the event times and the censoring times from {the density functions in~\eqref{eqn:f} and~\eqref{eqn:g}.}
Confidence intervals are calculated at the point $x_0=2.5$ using 1000 sets of data and the bandwidth in the case of undersmoothing is $b=c_{opt}(2.5)n^{-1/4}$,
where $c_{opt}(2.5)=4$.
When estimating the bias we use $b=c_{opt}(2.5)n^{-5/17}$ to estimate the hazard and $b_1=c_{opt}(2.5) n^{-1/17}$ to estimate its second derivative (as suggested in~\citet{Hall92}).
Table~\ref{tab:2} shows the performance, for various sample sizes, of the confidence intervals based on the asymptotic distribution (AD)
of the Grenander-type estimator and of the smoothed Grenander estimator (for both undersmoothing and bias estimation).
\begin{table}
\begin{tabular}{ccccccc}
 \toprule
      & \multicolumn{2}{c}{Grenander}      &   \multicolumn{2}{c}{SG-undersmoothing} &   \multicolumn{2}{c}{SG-bias estimation}  \\
n     & AL    & CP    & AL    & CP    & AL    & CP     \\
50    & 0.157 & 0.822 & 0.129 & 0.948 & 0.136 & 0.929 \\
100   & 0.127 & 0.856 & 0.101 & 0.954 & 0.109 & 0.959  \\
500   & 0.073 & 0.859 & 0.056 & 0.971 & 0.064  & 0.979  \\
1000  & 0.058 & 0.864 & 0.043 & 0.979 & 0.050 & 0.976  \\
5000  & 0.032 & 0.845 & 0.023 & 0.965 & 0.029 & 0.965  \\
\bottomrule
\end{tabular}
\caption{The average length (AL) and the coverage probabilities (CP) for $95\%$ pointwise confidence intervals of the density function at the point $x_0=2.5$ based on the asymptotic distribution. }
\label{tab:2}
\end{table}
Again, confidence intervals based on the Grenander-type estimator have a poor coverage.
On the other hand, by considering the smoothed version, we usually obtain high coverage probabilities. 
Again, undersmoothing behaves slightly better.
The performance of these three methods for a fixed sample size $n=500$ and different points of the support of the density is illustrated in Figure~\ref{fig:subfig5}.
\begin{figure}
\centering
\subfloat[][]
{\includegraphics[width=.48\textwidth]{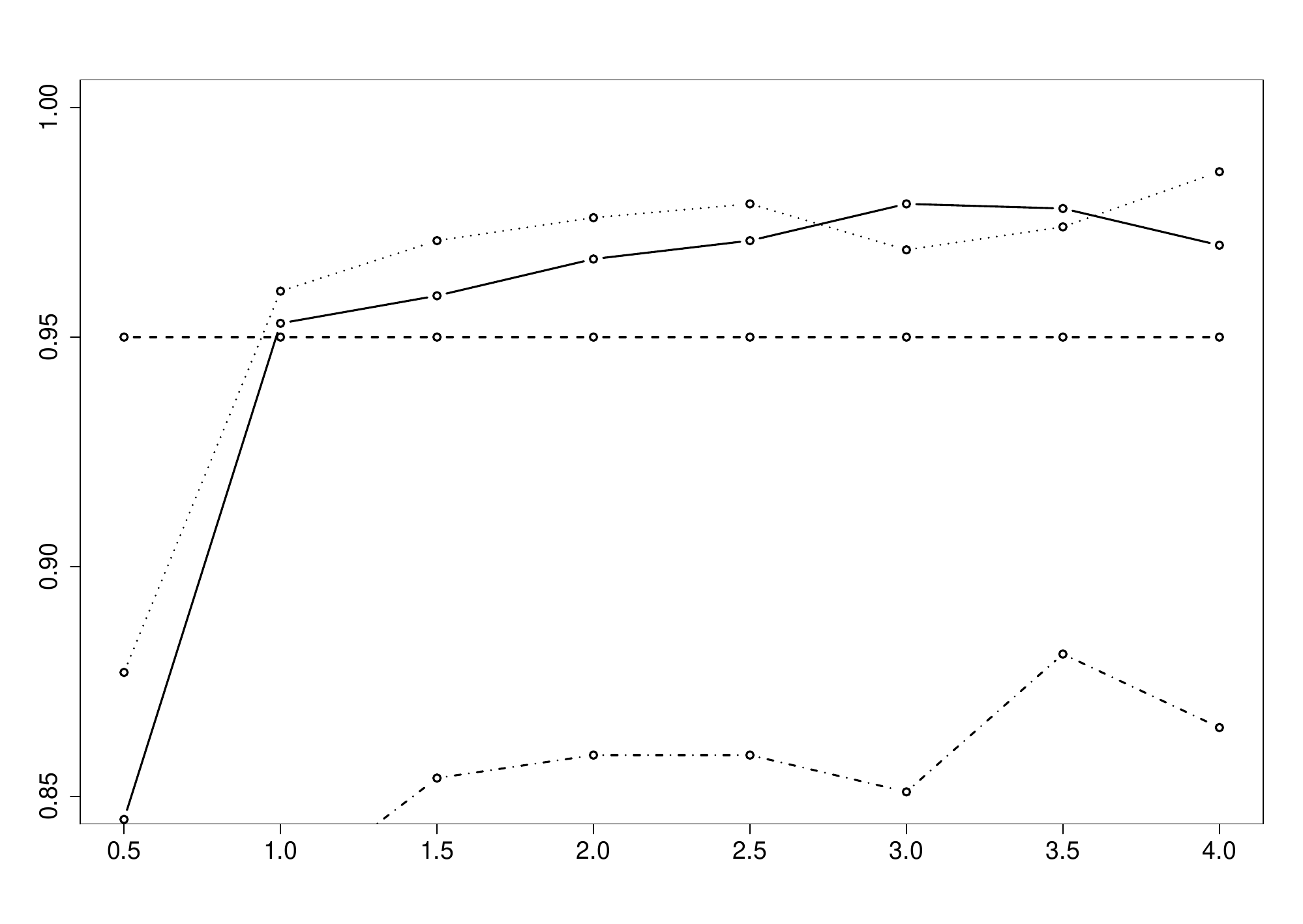}} \quad
\subfloat[][]
{\includegraphics[width=.48\textwidth]{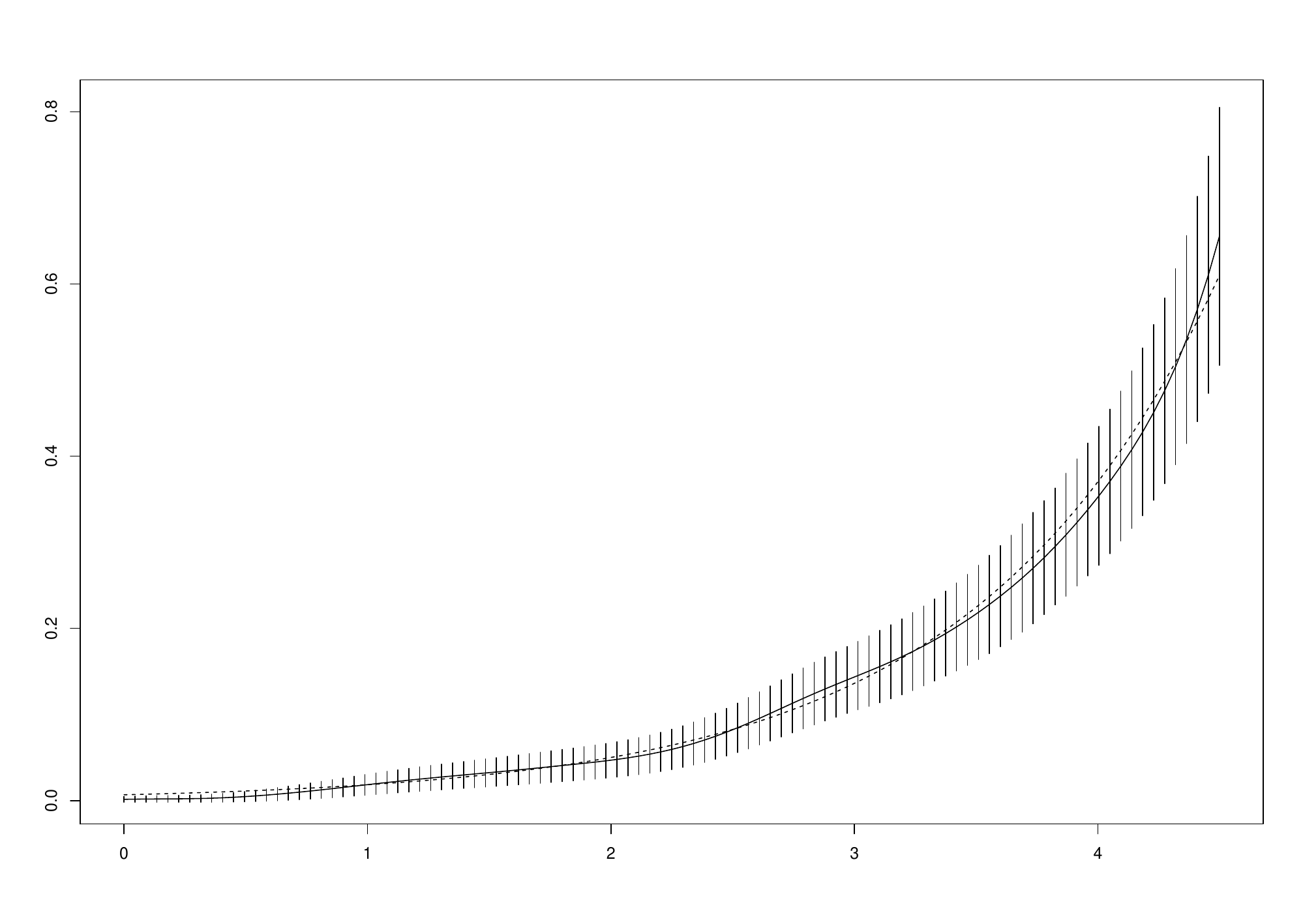}}
\caption{Left panel: Actual coverage of confidence intervals with nominal coverage $95\%$ for the density function based on the asymptotic distribution. Dashed line-nominal level; dotdashed line-Grenander-type; solid line-SG undersmoothing; dotted line-SG bias estimation. Right panel: $95\%$ confidence intervals based on the asymptotic distribution for the density function using undersmoothing.}
\label{fig:subfig5}
\end{figure}

\section{Discussion}\label{sec:discussion}
In the present paper, we have considered kernel smoothed Grenander estimators for a monotone hazard and a monotone density under right censoring.
We have established uniform strong convergence of the estimators in the interior of the support of distribution of the follow-up times
and  asymptotic normality at a rate of convergence of $n^{2/5}$.
The behavior of the estimators has been illustrated in a small simulation study, where it can be seen that the smoothed versions perform better than the ordinary Grenander estimators.
The proof of asymptotic normality is more or less straightforward thanks to a Kiefer-Wolfowitz type of result provided in \citet{DL14}.

The right censoring model is a special case of the Cox regression model, where in addition one can consider various covariates.
A natural question then is whether the previous approach for proving the asymptotic normality of a smoothed Grenander-type estimator, for example, for the hazard rate can be extended to such a more general setting.
Unfortunately, no Kiefer-Wolfowitz nor an embedding into the Brownian motion is available for the Breslow estimator, being the natural naive estimator for the cumulative hazard.
Recently, \citet{GJ14} developed different approach to establish asymptotic normality of smoothed isotonic estimators,
which is mainly based on uniform $L_2$-bounds for the distance between the non-smoothed estimator and the true function.
This approach seems to have more potential for generalizing our results to the Cox model.
However, things will not go smoothly, because the presence of covariates makes it more difficult to obtain bounds on the tail probabilities
for the inverse process involved and in the Cox model one has to deal with a rather complicated martingale associated with the Breslow estimator.
This is beyond the scope of this paper, but will be the topic of future research.

%\section*{References}
\bibliography{shapeconstrained-estimation}

\end{document}